\documentclass[11pt]{amsart}
\usepackage{amssymb, latexsym, amsmath, amsfonts, tikz}
\usepackage{hyperref, enumitem, fancyhdr}
\usepackage{color}
\usepackage{tikz-cd} 

\newtheorem{thm}{Theorem}[section]
\newtheorem{cor}[thm]{Corollary}
\newtheorem{lem}[thm]{Lemma}

\theoremstyle{definition}

\theoremstyle{remark}
\newtheorem{rem}[thm]{Remark}
\numberwithin{equation}{section}
\theoremstyle{remark}

\newcommand{\no}{\noindent}

\begin{document}
\title{On automorphisms of super-level sets of Green's functions of H\'{e}non maps} 
%\keywords{H\'{e}non maps, escaping sets, automorphism}
\author{Mahima and Ratna Pal}

\address{M: Indian Institute of Science Education and Research Mohali, Knowledge City, Sector -81, Mohali, Punjab-140306, India}
\email{ph22080@iisermohali.ac.in, mahimamath06@gmail.com}
\address{RP: Theoretical Statistics and Mathematics Unit,
Indian Statistical Institute, Bangalore Centre, 8th Mile Mysore Road, Bangalore, 560059.}
\email{ratna.math@gmail.com, ratnapal@isibang.ac.in}

\begin{abstract}
The aim of this article is two-fold. First, we obtain a normal form for automorphisms of the escaping sets of H\'{e}non maps in a neighborhood of the attracting fixed point at infinity. Building on this local description, we characterize the global automorphisms of $\mathbb{C}^2$ that preserve the escaping sets. The analytic structure of the escaping sets, as established by Hubbard--Oberste-Vorth (\cite{HOV}), plays a crucial role in the proof of these results.

\medskip 
\noindent
In a similar spirit, we investigate the analytic structure of the super-level sets of the Green's functions associated with H\'{e}non maps and give description of the automorphisms of these domains.
\end{abstract}

\maketitle

\section{Introduction}

In this article, we investigate several rigidity properties of H\'{e}non maps. These are polynomial automorphisms of $\mathbb{C}^2$ of the form 
\begin{equation} \label{Henon def}
H(x,y)=(y,p(y)-\delta x), 
\end{equation}
where $(x,y)\in \mathbb{C}^2$, $p$ is a polynomial of degree $d\geq 2$, and $\delta\neq 0$. Up to conjugacy, H\'{e}non maps constitute the only class of dynamically interesting polynomial automorphisms (see \cite{FM}), and consequently play a central role in higher dimensional holomorphic dynamics. 

\medskip 
\noindent
The present article is concerned with two dynamically significant domains associated with H\'{e}non maps and the corresponding rigidity phenomena: the {\it escaping sets} of H\'{e}non maps and the {\it super-level sets} of the {\it Green's functions} of H\'{e}non maps. Recall that escaping set (forward) of a given H\'{e}non map $H$ is defined by
\[
U^{+}_H =\left \{(x, y) \in \mathbb C^2 : \lVert H^n(x,y)\rVert \rightarrow \infty\right\},
\]
while the associated Green's function (forward) is given by 
\begin{equation}\label{green def}
    G_H^+(x,y)=\lim_{n\to \infty} \frac{1}{d^n}\log^+\lVert H^n(x,y)\rVert,
\end{equation}
where $\log^+ \vert y\vert =\max\{\log \vert y\vert ,0\}$. We next introduce several additional dynamical objects that play an essential role in  the rigidity questions considered in this article. The {\it non-escaping set} (forward) of a H\'{e}non map $H$ is defined by 
\[
K^{+}_H =\left \{(x, y) \in \mathbb C^2 : {\{H^n(x,y)\}}_{n\geq 1} < \infty
\right\}.
\]
It can be verified easily that for a given H\'{e}non map, the escaping sets and the non-escaping sets are complements of one another. 
The {\it Julia set} $J_H^+$ is defined as the boundary of the non-escaping set $K_H^+$, namely, 
$$
J_H^+= \partial K_H^+.
$$ 
A standard normal family argument shows that the dynamical instability of H\'{e}non maps take place precisely on its Julia set. An alternative and particularly useful description of the escaping sets and the non-escaping sets is provided by the Green's function. More precisely,  
$$
U_H^+=\{G_H^+>0\},
$$
whereas
$$
K_H^+=\{G_H^+=0\}.
$$ 

\medskip 
\noindent
The first  rigidity question addressed in this article is the following: given an automorphism $f$ of the escaping set $U_H^+$ of a H\'{e}non map $H$, can we obtain an explicit description of $f$? This question can be viewed of as a natural extension of the main results in \cite{BPV} and \cite{Pal1}. In \cite{BPV}, the authors obtain an explicit description of automorphisms $f$ of $\mathbb{C}^2$ satisfying  
$$
f(U_H^\pm)=U_H^\pm. 
$$ 
Here the set $U_H^-$ can be defined analogously as $U_H^+$, i.e.,
\[
U_H^-=\{(x,y)\in \mathbb{C}^2: \|H^{-n}(x,y)\|\to \infty\}.
\]
In \cite{Pal1}, an explicit description of $f$ satisfying
 $$
 f(U_H^+)=U_H^+
 $$ is obtained when $f$ is a polynomial automorphism of $\mathbb{C}^2$. The methods employed in these works are essentially global in nature and therefore cannot be applied when $f$ is assumed to be only an automorphism of the escaping set.  In contrast, the rigidity problems considered here is closer in spirit to the unpublished work of Bousch \cite{Bousch}, which investigates the automorphism group of the escaping set of quadratic H\'{e}non maps, as well as to the results  \cite{BRT} and \cite{Pal} giving relation between H\'{e}non maps sharing biholomorphic escaping sets. 
 
\medskip 
\noindent  
 A detailed understanding of the analytic structure of the escaping set $U_H^+$ developed by Hubbard and Oberste-Vorth, plays a fundamental role in all of these works. We briefly recall their construction. The fundamental group of $U_H^+$ is 
 $$
\pi_1(U_H^+)=\mathbb{Z}[1/d], 
 $$
 and  the escaping set $U_H^+$ can be realized as a quotient of 
$\mathbb{C}\times (\mathbb{C}\setminus{\overline{\mathbb{D}}})$ by a discrete subgroup of its automorphism group isomorphic to ${\mathbb{Z}[1/d]}/ {\mathbb{Z}}$  (see \cite{HOV}). The construction relies crucially on the B\"{o}ttcher function
\[
\phi_H^+:V_R^+ \rightarrow \mathbb{C}\setminus \overline{\mathbb{D}},
\]
where $R>0$ is large enough, and
% were used crucially to come up with this description of $U_H^+$.  Note that the B\"{o}ttcher function $\phi_H^+$ is defined only on a certain neighbourhood of $[0:1:0]$, namely on 
\[
V_R^+=\left\{(x,y)\in \mathbb{C}^2 : \lvert y \rvert>\max\{R, \lvert x\rvert\}\right\} \subset U_H^+.
\]
Although $\phi_H^+$ does not extend analytically to all of $U_H^+$,  it admits analytic continuation along every curve in $U_H^+$ starting in $V_R^+$. Consequently, it defines a multi-valued analytic function on $U_H^+$. Let $\hat{U}_H^+$ be the Riemann domain of this multi-valued function, and let 
$$
\hat{\phi}_H^+: \hat{U}_H^+\rightarrow \mathbb{C}\setminus \overline{\mathbb{D}}.
$$
be the single-valued analytic function, which arises by considering infinitely many branches of $\phi_H^+$. 
It follows from the construction of Hubbard and Oberste--Vorth that $\hat{U}_H^+$ is biholomorphically equivalent to the domain $\{(z,\zeta)\in \mathbb{C}^2:|\zeta|>1\}$. Further, in this new coordinate each copy of the complex plane is obtained by straightening out the level set $\{\hat{\phi}_H=\zeta\}$. 

\medskip 
\noindent
The above description enable us to establish the following normal form for automorphisms of the escaping set. 
\begin{thm}\label{Thm2}
Let $H$ be a H\'{e}non map defined in (\ref{Henon def}). Let 
$$
f=(f_1,f_2):U_H^+ \to U_H^+
$$
be an automorphism inducing identity on $\pi_1(U_H^+)$, then for $0<\epsilon<1$, and a large $\rho_1>1$, 
\[
f(x,y)=\left(a_1 x+b_1 y+c_1+\sum_{n=1}^{\infty}\frac{\alpha_n(x)}{y^n},a_2 x+b_2 y+c_2+\sum_{n=1}^{\infty}\frac{\beta_n(x)}{y^n}\right),
\]
on $\mathcal{B}_1=\{(x,y)\in \mathbb{C}^2: |x|<\epsilon, |y|>2\rho_1\}$, where $a_1,a_2,b_1,b_2,c_1,c_2 \in \mathbb{C}$, and for $n\geq 1$, $\alpha_n$ and $\beta_n$ are polynomials in one variable $x$ of degree at most $n+1$. 
\end{thm}
 
\noindent
When $f$ extends to an automorphism of $\mathbb{C}^2$, that is, when 
$$
f\in {\rm{Aut}}(U_H^+)\cap {\rm{Aut}}(\mathbb{C}^2),
$$
we obtain an explicit description of $f$.  In this setting, we do not require the assumption that $f$ induces the identity on $\pi_1(U_H^+)$. The latter hypothesis in Theorem~\ref{Thm2} is purely technical and is needed to lift $f$ to an automorphism of
$
\mathbb C\times\left(\mathbb C\setminus\overline{\mathbb D}\right).
$
When $f$ is already a global automorphism of $\mathbb C^2$, this lifting assumption can be bypassed by using existing rigidity results (see the proof of Theorem~\ref{Thm1} for details).

\begin{thm}\label{Thm1}\label{global aut escaping}
Let $H$  be a H\'{e}non map of degree $d\geq 2$ as defined in (\ref{Henon def}).  Let $f$ be an automorphism of $\mathbb{C}^2$ such that $f$ induces an automorphism of $U_H^+$. 
Then $f$ is a polynomial automorphism such that
\[
f^{s_1}=L\circ H^{s_2},
\] for some $s_1, s_2 \in \mathbb{Z}$ and linear map $L$ of the form $(x,y) \rightarrow (a x,a^d y)$ with $a^{d^2-1}=1$. 
\end{thm} 

Theorem~\ref{Thm1} strengthens a rigidity result from an earlier version of \cite{BV}, where the corresponding statement was proved for a restricted class of H\'{e}non maps. During the preparation of this article, the authors of \cite{BV} informed us of an updated version of their work. In its current form, their theorem establishes the same conclusion as Theorem~\ref{Thm1} for compositions of generalized H\'{e}non maps. Although the statements are closely related, our proof is substantially different from that in \cite{BV}. Our approach is based on the observation that whenever an automorphism of $U_H^+$ admits a lift to the covering $\hat{U}_H^+$, it preserves a suitable neighborhood of the point $[0:1:0]$ in $\mathbb P^2$. This allows us to determine the precise form of the automorphism. In contrast, the proof in \cite{BV} relies on the fact that the subgroup
$
\mathrm{Aut}(U_H^+)\cap\mathrm{Aut}(\mathbb C^2)
$
consisting of automorphisms admitting lifts to
$
\mathbb C\times\left(\mathbb C\setminus\overline{\mathbb D}\right)
$
is a finite cyclic group whose order is determined by the degree of the underlying H\'{e}non map.

\medskip 
\noindent
We now record several consequences of Theorem~\ref{Thm1}.

%We set out to prove Theorem \ref{Thm1} to strengthen a rigidity result stated in an earlier version of \cite{BV}, which proves Theorem \ref{Thm1} for those H\'{e}non maps for which. During the preparation of the present article we were kindly informed by the authors of \cite{BV} about their improved version. Theorem in \cite{BV}  states the same as Theorem \ref{Thm1} for composition of generalised H\'{e}non maps.  
%Next we record a few interesting corollaries of Theorem \ref{Thm1}. However, note that the main idea of our proof deviates from the one in \cite{BV}. Our proof relies on the fact that whenever $f$ is an automorphism of $U_H^+$ which can be lifted as an automorphism of $\tilde{U}_H^+$,  certain neighbourhood of $[0:1:0]$ in $\mathbb{C}^2$ remains invariant under $f$. This fact enables us to get hold on the form of $f$. On the other hand, the proof in \cite{BV} is built on the fact that the subgroup $\rm{Aut}(U_H^+)\cap \rm{Aut}(\mathbb{C}^2)$ whose elements can be lifted to $\mathbb{C}\times (\mathbb{C}\setminus{\bar{\mathbb{D}}})$ is a finite cyclic group whose order is determined by the degree of the underlying H\'{e}non maps. 

\begin{cor}\label{Cor1}
Let $H$ and $F$ be a pair of H\'{e}non maps. Let $K_H^+$ and $K_F^+$ be the forward non-escaping sets of $H$ and $F$, respectively.  Let $f$ be an  automorphism of $\mathbb{C}^2$ such that $f(K_H^+)=K_F^+$, then $f$ is a polynomial automorphism such that  
\[
F^{s_1}\circ f=f \circ L\circ H^{s_2}, 
\] 
for some linear map $L(x,y)=(a x,a^d y)$ with $a^{d^2-1}=1$ and $s_1, s_2\in \mathbb{Z}$ of same sign.  
\end{cor}

\begin{cor}\label{Cor2.1}
    Let $H$ be a H\'{e}non map of degree $d\geq 2$, and let $f\in \mathrm{Aut}(\mathbb{C}^2)$ such that $f(J_H^+)=J_H^+$. Then, $f$ is a polynomial automorphism such that 
    \[
    f^{s_1}=L\circ H^{s_2}
    \]
    for $s_1,s_2\in \mathbb{Z}$, and $L(x,y)=(a x,a^d y)$, where $a^{d^2-1}=1$.
\end{cor}

\begin{cor}\label{Cor2.2}
    Let $H$ and $F$ be a pair of H\'{e}non maps. Let $f\in \mathrm{Aut}(\mathbb{C}^2)$ such that $f(J_H^+)=J_F^+$. Then, $f$ is a polynomial automorphism such that 
    \[
F^{s_1}\circ f=f\circ L \circ H^{s_2}, 
\] 
for some linear map $L(x,y)=(a x,a^d y)$ with $a^{d^2-1}=1$ and $s_1,s_2\in \mathbb{Z}$. 
\end{cor}

\noindent
For $c>0$, let 
$$
\Omega^c=\{(x,y)\in \mathbb{C}^2: G_H^+(x,y)>c\}.
$$
As in \cite{HOV}, we show that the fundamental group of $\Omega^c$ is $\mathbb{Z}[1/d]$. Further,  we construct covering of $\Omega^c$ corresponding to $\mathbb{Z}$, which turns out to be biholomorphic to $\mathbb{C}\times \mathcal{A}^c$, where $\mathcal{A}^c=\left\{t\in \mathbb{C}: \lvert t\rvert>\exp(c)\right\}$ (see Section 3). Using this analytic structure obtained for $\Omega^c$, we
obtain the following theorem.   

\begin{thm}\label{Thm4}
Let $H$ be a H\'{e}non map defined in \ref{Henon def}. For $c_1$, $c_2>0$, let $f=(f_1,f_2):\Omega^{c_1}\to \Omega^{c_2}$ be a biholomorphism which induces identity on $\mathbb{Z}[1/d]$, then $c_1=c_2$. Further, 
\begin{itemize}
\item[1.] 
for $0<\epsilon<1$, and a large $\rho_1>1$, 
\[
f(x,y)=\left(a_1 x+b_1 y+c_1+\sum_{n=1}^{\infty}\frac{\alpha_n(x)}{y^n},a_2 x+b_2 y+c_2+\sum_{n=1}^{\infty}\frac{\beta_n(x)}{y^n}\right),
\]
on $\mathcal{B}_1=\{(x,y)\in \Omega^{c_1}: |x|<\epsilon, |y|>2\rho_1\}$, where $a_1,a_2,b_1,b_2,c_1,c_2 \in \mathbb{C}$ and for $n\geq 1$, $\alpha_n$ and $\beta_n$ are polynomials in one variable $x$ of degree at most $n+1$; 
\item[2.]
for every $c_3>c_1$, $f:\Omega^{c_3}\to \Omega^{c_3}$ is an automorphism;
\item[3.] if $f\in \mathrm{Aut}(\mathbb{C}^2)$, then $f$ is a linear map of the form 
    \[
    f(x,y)=(a x,a^d y),
    \]
    where $a^{d^2-1}=1$, and $f(K_H^\pm)=K_H^\pm$.
\end{itemize}
\end{thm}

\noindent
This article is organized as follows. In Section~2, we introduce the basic terminology and recall several known results that will be used throughout the paper. In Section~3, we construct an intermediate covering of $\Omega^c$. In Section~4, we give an explicit description of the lifts of automorphisms of $\Omega^c$ and $U_H^+$ to their respective covering spaces. In Section~5, we give proofs of Theorem \ref{Thm2} and Theorem \ref{Thm1} using the explicit forms of these lifts. In Section~6, we establish the corollaries stated above.  In Section~7, we give proof of Theorem \ref{Thm4}. 

\medskip 
\noindent
{\textbf{Acknowledgement:}}  
The first author is supported by University Grants Commission (UGC). The second author is partially supported by the Mathematical Research Impact Centric Support\\
 (MTR/2023/001258) and the Advanced Research Grant (ARG/2025/007852) of the Anusandhan National Research Foundation (ANRF), India. A substantial part of this work was carried out during the first author's visit to the Indian Statistical Institute (ISI), Bangalore Centre. The first author sincerely thanks the Indian Statistical Institute, Bangalore Centre, for its warm hospitality and excellent research environment during her visit.

\section{Preliminaries}
\noindent 
In this section, we present rudimentary terminologies  and a few known fundamental results pertinent to studying dynamics of H\'{e}non maps. 
 
 \medskip 
 \noindent
We consider a H\'{e}non map $H$ as defined in \eqref{Henon def}. Conjugating $H$ with a suitable affine map one can consider the polynomial $p$ to be a monic and centered of degree $d \geq 2$, i.e.,  $p(y)=y^d+c_{d-2}y^{d-2}+\cdots+c_0$. For $R>0$, we define 
\begin{align*}
    V_R^+ & =\{(x,y)\in \mathbb{C}^2 : \max\{|x|,R\}<|y|\},\\
    V_R^- & =\{(x,y) \in \mathbb{C}^2 : \max\{|y|,R\}<|x|\},\\
    V_R & = \{(x,y)\in \mathbb{C}^2 : \max\{|x|, |y|\}\leq R\}.
\end{align*}
One can choose a large $R>0$ such that 
\[
V_R^+\subset H^{-1}(V_R^+)\subset H^{-2}(V_R^+)\subset \cdots,
\]
\[
U_H^+ =\bigcup_{n\geq 0} H^{- n}\left(V_R^+\right) \text{ and } K_H^+\subset V_R\cup V_R^-.
\]
The above trapping properties of the sets $\{V_R^\pm, V_R\}$ helps localizing the dynamics of H\'{e}non maps. As in dimension one, the method of pluri-potential theory plays a significant role in understanding the iteration theory of H\'{e}non maps. It turns out that the pluri-complex  Green's function of the non-escaping set $K_H^+$ is the Green's function $G_H^+$ of the H\'enon map $H$ as defined in \eqref{green def}. 
The Green's function $G_H^+$ is  plurisubharmonic on $\mathbb C^2$, non-negative everywhere, and pluriharmonic on $\mathbb C^2 \setminus K^{+}_H$. Further, $G_H^+$ vanishes precisely on $K^{+}_H$. By construction, 
\[
G^{+}_H \circ H = d G^{+}_H.
\]
Moreover, $G_H^+$ have logarithmic growth near infinity, i.e., there exist $R>0$ and $C>0$ such that   
\begin{equation}\label{L1}
\log^+ \lvert y \rvert-C\leq G_H^+ (x,y) \leq \log^+ \lvert y \rvert+C,
\end{equation}
for all $(x,y)\in \overline{V_R^+ \cup V_R}$.

\medskip 
\no 
Each H\'{e}non map $H$ extends meromorphically to $\mathbb{P}^2$ with an attracting fixed point at $I^- = [0:1:0]$ and indeterminacy point at $I^+ = [1:0:0]$. Since $I^-$ is an attracting fixed point of $H$, in a similar spirit as for polynomials in the complex plane, one can construct B\"ottcher coordinates for H\'{e}non maps as 
\[
\phi_H^+(x,y)=\lim_{n\to \infty} y. \left(\frac{y_1}{y^d}\right)^{\frac{1}{d}}. \cdots . \left(\frac{y_{n+1}}{y_n^d}\right)^{\frac{1}{d^{n+1}}}.\cdots, 
\]
for $(x,y)\in V_R^+$. Here $y_n$'s are the projections of $H^n(x,y)$ to the $y$-axis. It turns out that $\phi_H^+:V_R^+\to \mathbb{C}\backslash\mathbb{\overline{D}}$ is a holomorphic function such that 
\begin{equation}\label{Bottcher}
\phi_H^+\circ H=(\phi_H^+)^d,
\end{equation}
and 
\[
\phi_H^+(x,y)\sim y \text{ as $||(x,y)||\to \infty$ in } V_R^+.
\]
The readers may see \cite{BS1}, \cite{HOV}, \cite{MNTU} for detailed construction and properties of Green's functions and B\"{o}ttcher coordinates. 

\medskip
\noindent
Next we discuss analytic structure of the escaping set $U_H^+$ due to Hubbard and Oberste-Vorth (\cite{HOV}, \cite{MNTU}). In Section 3, we give a similar construction for super-level sets of Green's function. In Section 4, we employ the analytic structure of $U_H^+$ directly for obtaining description of lifts of automorphisms of $U_H^+$. 

\medskip 
\noindent
For a H\'{e}non map of degree $d$, the fundamental group of $U_H^+$ turns out to be $\mathbb{Z}[1/d]=\{m/d^n:m,n\in \mathbb{Z}\}$. This follows from the fact that any loop in $V_R^+$ is homotopic to a loop of the form $C_m:t \mapsto \left(0,R e^{2\pi i m t}\right)$, for some $m\in \mathbb{Z}$, and any loop in $U_H^+$ can be pushed to $V_R^+$ applying sufficiently large iteration of $H$.  Formally this fact can be established by looking at the correspondence: 
\[
\alpha_H^+(C)=\frac{1}{2\pi i}\int_C \omega,
\]
where $C$ is a loop in $U_H^+$ and $\omega$ is the one form on $U_H^+$ obtained by extending the one form $\mathrm{d}\phi_H^+ / \phi_H^+$ using the identity $\phi_H^+\circ H=(\phi_H^+)^d$ on $V_R^+$. 

\medskip 
\noindent
The map $\phi_H^+$ is an analytic function on $V_R^+$ and it cannot be extended holomorphically to $U_H^+$. %Nevertheless, $\phi_H^+$ extends along any curve in $U_H^+$ starting at a point in $V_R^+$ and
But $\phi_H^+$ can be extended as an infinite valued analytic map $\hat{\phi}_H^+$ on its Riemann domain $\hat{U}_H^+$. A formal construction of $\hat{U}_H^+$ can be obtained as follows.  Fix $a\in V_R^+$, and define 
\[
\hat{U}_H^+ = \left\{(z,C): z\in U_H^+,  C \text{ is a path from } a \text{ to } z \text{ in } U_H^+ \right\}/\sim_H,
\]
where $\sim_H$ is equivalence relation defined on the set $\{(z,C): z\in U_H^+, C \text{ is a path from } a \text{ to } z \text{ in } U_H^+\}$ by
\[
(z,C)\sim_H(z',C') \iff z=z' \text{ and } \alpha_H^+(C\bar{C}')\in \mathbb{Z},
\]
where $\bar{C}'$ is obtained by taking the curve $C'$ with the reverse orientation. 
Note that $\hat{U}_H^+$ is nothing but the covering of $U_H^+$ corresponding to the subgroup $\mathbb{Z} \subseteq \mathbb{Z}[1/d]$. Let $\hat{V}_R^+=\{[z,C]\in\tilde{U}_H^+ : C\subset V_R^+\}$ and let $\pi_H$ be the covering map, i.e.,  $\pi_H: [z,C]\mapsto z$.  Let $\hat{\phi}_H^+:\hat{U}_H^+ \to \mathbb{C}\backslash\mathbb{\overline{D}}$ be defined by
\[
 [z,C]\mapsto \phi_H^+(a)\left(e^{\int_C\omega}\right).
 \]
A simple calculation yields $\pi_H|_{\hat{V}_R^+}:\hat{V}_R^+\to V_R^+$ is a biholomorphism, and $\hat{\phi}_H^+|_{\hat{V}_R^+}\equiv \phi_H^+|_{V_R^+}$.

\medskip 
\noindent
For large constant $M>0$, let 
\[
U_R^+=\left\{(x,y)\in V_R^+ : \left\lvert \phi_H^+(x,y)\right\rvert>M\max\{R,|x|\}\right\}.
\]
Note that $H(U_R^+)\subseteq U_R^+$. The following lemma is crucial in obtaining an explicit description of $\hat{U}_H^+$ (see \cite[Lemma 7.3.7]{MNTU}).  We only give a brief sketch of the proof for the first part of the lemma. However a detailed proof is provided for the later part since it differs from an analogous statement presented in \cite[Lemma 7.3.7]{MNTU}.
\begin{lem}\label{prop2.2}
There exist a holomorphic function $\psi_H^+:U_R^+\to\mathbb{C}$ such that
\begin{equation}\label{psi}
\psi_H^+\circ H(x,y)=\frac{\delta}{d}\psi_H^+(x,y)+Q\left (\phi_H^+(x,y)\right)
\end{equation}
where $Q(t)=t^{d+1}+b_{d-1}t^{d-1}+\cdots+b_0$ is a monic polynomial in single variable of degree $d+1$, and 
\begin{equation}\label{Phi}
\Phi_H^+=(\psi_H^+,\phi_H^+):U_R^+\to \mathbb{C}\times \mathbb{C}\backslash\mathbb{\overline{D}}
\end{equation}
is an injective holomorphic map. Further,  there exist a sufficiently large constants $\rho_1$  and a sufficiently small constant $\rho_2$ such that the map $\Phi_H^+=\left(\psi_H^+,\phi_H^+\right):U_R^+\to \mathbb{C}^2$ contains the domain $\{(s,t)\in \mathbb{C}^2 : |s|<\rho_2|t|^2, |t|>\rho_1\}$ in its range set.
\end{lem}

\begin{proof}
Since $\phi_H^+(x,y)\sim y \text{ as $\lVert (x,y)\rVert \to \infty$ in } V_R^+$, the map 
\[
(x,y)\mapsto (x,t)=\left (x,\phi_H^+(x,y)\right)
\]
is a biholomorphism from $U_R^+$ onto its image. Let
\[
(x,t)\mapsto \left(x,\lambda_H^+(x,t)\right),
\] 
be its inverse. For $(x,y)\in U_R^+$, let
\[
s_1=\phi_H^+(x,y) \int_0^x \frac{\partial\lambda_H^+}{\partial t}(\xi,t)\mathrm{d}\xi ,
\]
where $t=\phi_H^+(x,y)$. Then the H\'{e}non maps $H$ in $(s_1,t)$ coordinate is given by
\[
(s_1,t)\mapsto \left(\frac{\delta}{d}s_1+Q_1(t),t^d\right),
\]
where $Q_1(t)$ is holomorphic for $|t|>RM$. Further an explicit expression of $Q_1(t)$ can be given as: 
\[
Q_1(t)=t^d\int_0^{\lambda_H^+(0,t)}\frac{\partial\lambda_H^+}{\partial t}(\xi,t^d)\mathrm{d}\xi.
\]
Since $\lambda_H^+(0,t)-t$ is bounded, $\left\lvert Q_1(t)\right \rvert\leq {\lvert t \rvert}^{d+1}$. Thus 
\[
Q_1(t)=t^{d+1}+b_{d}t^d+b_{d-1}t^{d-1}+\cdots+b_0+\frac{c_1}{t}+\frac{c_2}{t^2}+\cdots,
\]
with $b_j,c_i\in \mathbb{C}$, for $0\leq j \leq d$ and $i\geq 1$.
Let $Q$ be the polynomial part of the Laurent expansion of $Q_1$ and let
\begin{align*}
r(t) & = \frac{d}{\delta}(Q_1(t)-Q(t))+\left(\frac{d}{\delta}\right)^2(Q_1(t^d)-Q(t^d))+\cdots\\
& =\frac{d}{\delta}\left(\frac{c_1}{t}+\frac{c_2}{t^2}+\cdots\right)+\left(\frac{d}{\delta}\right)^2\left(\frac{c_1}{t^d}+\frac{c_2}{t^{2d}}+\cdots\right)+\cdots.
\end{align*}
Let
\begin{equation}\label{psi_H^+}
\psi_H^+(x,y)=s=\phi_H^+(x,y) \int_0^x \frac{\partial\lambda_H^+}{\partial t}(\xi,t)\mathrm{d}\xi +r(t).
\end{equation}
Then H\'{e}non map $H$ in $(s,t)$ coordinate is given by
\[
(s,t)\mapsto \left(\frac{\delta}{d}s+Q(t),t^d\right).
\]
Thus \eqref{psi} holds. Since for sufficiently large $R$ the map $(x,y)\mapsto(s_1,t)$ is a biholomorphism from $U_R^+$ onto its image, $\Phi_H^+$ is injective on $U_R^+$. The identity $\phi_H^+\left(0,\lambda^+_H(0,t)\right)=t$ implies that $b_d=0$.

\medskip
\noindent
Now we prove the last part of the lemma. Let $\epsilon_0>0$. Since $\phi_H^+(x,y)\sim y$, as $\left\lVert (x,y)\right\rVert\to \infty$ in $V_R^+$, there exists a sufficiently large $R$ such that $\frac{\partial\lambda_H^+}{\partial t}(x,t)\sim 1$, i.e.,
  \[
  (1-\epsilon_0)<\left|\frac{\partial\lambda_H^+}{\partial t}(x,t)\right|<(1+\epsilon_0),
 \]
 for $(x,t)\in\left\{(x,t)\in \mathbb{C}^2:\max\{M|x|,RM\}<|t|\right\}$.
Fix $|t|>MR$, and consider the function
\[
\lambda_t : x\mapsto \int_0^x \frac{\partial\lambda_H^+}{\partial t}(\xi,t)\mathrm{d}\xi.
\]
Note that $\lambda_t$ is holomorphic on $\{x\in \mathbb{C}: |x|<|t|/M\}$ and we could choose $M$ large such that $\lambda_t$ is continuous on the boundary $\{x\in \mathbb{C}: |x|=|t|/M\}$. Note that $\lambda_t(0)=0$ and for $|x|=|t|/M$, 
\[
(1-\epsilon_0)|t|/M<|\lambda_t(x)|.
\]
Let $x_1\in \mathbb{C}$ be such that $|x_1|<(1-\epsilon_0)|t|/M$ and consider the holomorphic function
\[
\lambda_{t,x_1} : x\mapsto \lambda_t(x)-x_1
\]
on $\{x:\mathbb{C}:|x|<|t|/M\}$. Then 
\[
\left\lvert\lambda_{t,x_1}(x)-\lambda_t(x)\right\rvert=|x_1|<|\lambda_t(x)|
\]
for $|x|=|t|/M$. Thus by Rouche's theorem, $\lambda_{t,x_1}$ and $\lambda_t$ have the same number of zeros in $\{x\in \mathbb{C}^2: |x|<|t|/M\}$. Thus there exists $x_0\in \{|x|<|t|/M\}$ such that $\lambda_{t,x_1}(x_0)=0$, i.e, $\lambda_t(x_0)=x_1$. Therefore, the disc $\{x\in \mathbb{C}^2 : |x|<(1-\epsilon_0)|t|/M\}$ is in the range of the function $\lambda_t$. Thus for $(x,y)\in U_R^+$, the map
\[
(x,y)\mapsto \left(t \int_0^x \frac{\partial\lambda_H^+}{\partial t}(\xi,t)\mathrm{d}\xi,t\right), 
\]
where $t=\phi_H^+(x,y)$, contains the set $\left\{(s,t)\in \mathbb{C}^2 : |s|<(1-\epsilon_0)|t|^2/M, |t|>MR\right\}$ in its range. Now choose $\rho_1>MR$ large enough such that for $|t|>\rho_1$, $|r(t)|<(1-\epsilon_0)|t|^2/2M$. Let $|s|<(1-\epsilon_0)|t|^2/2M$, then
\[
|s-r(t)|<(1-\epsilon_0)|t|^2/M.
\]
Thus there exists $(x,y)\in U_R^+$ such that $\Phi_H^+(x,y)=(s,t)$.

\end{proof}

\medskip
\noindent
Let $\hat{H}:\hat U_H^+\to \hat U_H^+$ be lift of $H$. Let $\hat{U}_R^+=\{[z,C]\in \hat{V}_R^+ : \pi_H([z,C])\in U_R^+\}$, then (see \cite[Lemma 7.3.8]{MNTU}) 
\begin{equation}\label{hatU_H}
    \hat{U}_H^+=\bigcup_{n\geq 0} \hat{H}^{-n}\left(\hat{U}_R^+\right).
\end{equation}
 Define 
 $$
 \hat{\psi}_H^+[z,C]=\psi_H^+(z), \text{  } \hat{\phi}_H^+[z,C]=\phi_H^+(z),
 $$
 and 
$$
\hat{\Phi}_H^+\left([z,C]\right)= \left(\psi_H^+(z),\phi_H^+(z)\right),
$$
for $[z,C]\in \hat{U}_R^+$.  One can extend $\hat{\phi}_H^+$ and $\hat{\psi}_H^+$ to $\hat{U}_H^+$ using the the functorial equations (\ref{Bottcher}) and (\ref{psi}), respectively. In turn, we can extend $\hat{\Phi}_H^+$ to $\hat{U}_H^+$. Finally, one can prove the following result. 
\begin{thm}
 The map $\hat{\Phi}_H^+:\hat{U}_H^+\to \mathbb{C}\times\mathbb{C\backslash\mathbb{\overline{D}}}$ is a biholomorphism.
\end{thm}
\noindent
The above theorem shows that the Riemann domain of the B\"{o}ttcher coordinate $\phi_H^+$ is biholomorphic to $\mathbb{C}\times \mathbb{C\backslash\mathbb{\overline{D}}}$. The results proved in this article can be considered as applications and generalizations of this fundamental theorem.

\section{Construction of an intermediate covering of $\Omega^c$}
\noindent 
In this section, we describe the analytic structure of super-level sets of the Green's function $G_H^+$, namely $\Omega^c=\{G_H^+>c\}$, for any $c>0$.  The construction of the covering of $\Omega^c$ follows the similar track as it is done in \cite{BPV1} for the punctured sub-level sets of the Green's function $\Omega_c=\{0<G_H^+<c\}$, for $c>0$.   

\begin{thm}\label{Thm3}
Let $H$ be a H\'enon map of degree $d$ as defined in (\ref{Henon def}). Then the fundamental group, $\pi_1(\Omega^c)$ of $\Omega^c$ is $\mathbb{Z}[1/d]$. Further, the cover $\hat{\Omega}^c$ corresponding to the subgroup $\mathbb{Z} \leq \mathbb{Z}[1/d]$ is biholomorphic to $\mathbb{C}\times\mathcal{A}^c$, where $\mathcal{A}^c=\{t\in \mathbb{C} : |t|>e^c\}$. Moreover, 

\begin{itemize}
\item[1.]
the lift $\tilde{H}_c:\mathbb{C}\times\mathcal{A}^c\to \mathbb{C}\times\mathcal{A}^{dc}$ of $H|_{\Omega^c}$ is of the form 
\[
\tilde{H}(s,t)=\left(\frac{\delta}{d} s + Q(t), t^d\right),
\]
where $Q(t)=t^{d+1}+b_{d-1}t^{d-1}+ \cdots + b_{0}$ is a monic and centered polynomial of degree $d+1$;

\item[2.]
for each element $[k/d^n] \in \mathbb{Z}[1/d]/\mathbb{Z}$, there exists a unique deck transformation $\gamma_{\frac{k}{d^n}}:\mathbb{C}\times\mathcal{A}^c \to \mathbb{C}\times\mathcal{A}^c$ defined as
\[
\gamma_{\frac{k}{d^n}}(s,t)=\left(s+\frac{d}{\delta}\sum_{l=0}^{\infty}\left(\frac{d}{\delta}\right)^l\left(Q\left(t^{d^l}\right)\right )-Q\left (\left(e^{\frac{2 k \pi i}{d^n}}t\right)^{d^l}\right),e^{\frac{2 k \pi i}{d^n}}t\right).
\]
\end{itemize}
\end{thm}
\begin{proof}
{\it Step 1:} In this step we show that $\Omega^c$ is connected. Let $w_1$, $w_2 \in \Omega^c$. Let $c_1>0$ be such that 
\[
\min\left\{G_H^+(w_1),G_H^+(w_2)\right\}>c_1>c.
\]
Now $G_H^+(x,y)\sim \log|y|$, as $|y|\to \infty$ in $V_R^+$. Thus, for $\epsilon>0$, there exists $R_\epsilon>1$ such that 
\[
\log^+|y|-\epsilon <G_H^+(x,y)<\log^+|y|+\epsilon,
\] 
for $(x,y)\in V_{R_\epsilon}^+$. Let $n_0>1$ be such that $H^{n_0}(w_1)$, $H^{n_0}(w_2) \in V_{R_\epsilon}^+$ and $2\epsilon < d^{n_0}(c_1-c)$. Without loss of generality, let $\lvert {\left (H^{n_0}(w_1)\right)}_2\rvert \leq \lvert (H^{n_0}(w_2))_2\rvert$, where ${((x,y))}_2=y$ and let $\sigma$ be a path in $V_{R_\epsilon}^+$ from $w_1$ to $w_2$ such that for every $(x,y)\in \sigma$,
\[
\left \lvert (H^{n_0}(w_1))_2\right\rvert \leq |y|.
\]
Now for $(x,y)\in \sigma$,
\begin{align*}
    G_H^+(x,y) & >\log^+|y|-\epsilon \geq \log^+|(H^{n_0}(w_1))_2|-\epsilon \\
    & > G_H^+(H^{n_0}(w_1))-2\epsilon >d^{n_0}c_1-d^{n_0}(c_1-c)=d^{n_0}c.
\end{align*}
Therefore, $H^{-n_0}(\sigma)$ is a path from $w_1$ to $w_2$ in $\Omega^c$.

\medskip 
\noindent
{\it Step 2:} In this step we show that the fundamental group of $\Omega^c$ is  $\mathbb{Z}[1/d]$.
Recall from Section 2 that for sufficiently large constant $M$,
\[
U_R^+ = \{(x,y)\in V_R^+ : \left\lvert \phi_H^+(x,y)\right \rvert > M \max\{R, |x|\}\}, 
\]
and \(H(U_R^+) \subset U_R^+\). There exist large $n\geq 1$ such that $H^n(V_R^+)\subset U_R^+$. Therefore, 
\[
U_H^+ = \bigcup_{n \geq 0} H^{-n}(U_R^+).
\]
Let $V^+_{R,c} = V_R^+\cap \Omega^c \neq \emptyset$ and $U^+_{R,c} = U_R^+\cap \Omega^c\subset V^+_{R,c}$. It is easy to see that for $n\geq 0$, 
\[H(V^+_{R,d^n c})\subset V^+_{R,d^{n+1}c} \text{ and } H(U^+_{R,d^n c})\subset U^+_{R,d^{n+1}c}.
\]
Thus we have the following increasing sequences of open sets 
\[
V^+_{R,c} \subset H^{-1}(V^+_{R,d c})\subset H^{-2}(V^+_{R,d^2 c})\cdots,\]
and 
\[
U^+_{R,c} \subset H^{-1}(U^+_{R,d c})\subset H^{-2}(U^+_{R,d^2 c})\cdots.
\]
Also, for $m, n\geq 0$, $H^n:\Omega^{d^m c} \to \Omega^{d^{n+m} c}$ is a biholomorphism and 
\begin{equation}\label{eq6.1}
\Omega^c = \bigcup_{n\geq 0}H^{-n}\left(V^+_{R,d^n c}\right)=\bigcup_{n\geq 0}H^{-n}\left(U^+_{R,d^n c}\right).
\end{equation}
\noindent
Consider the closed 1-form defined on $V_R^+$ by 
$
\omega={\mathrm{d}\phi_H^+}/{\phi_H^+}.
$
This 1-form can be extended to $U_H^+$ using the functional equation $\phi_H^+\circ H=(\phi_H^+)^d$ (see \cite[Lemma 7.3.2]{MNTU}). One can also prove that $H^*(\omega)=d\omega$. For a closed path $C$ in $\Omega^c$, define
\[
\alpha(C)=\frac{1}{2\pi i}\int_C \omega.
\]
If $C$ is a path in $V^+_{R,c}$, then 
\[
\alpha(C)=\frac{1}{2\pi i}\int_C \frac{\mathrm{d}\phi_H^+}{\phi_H^+}=\frac{1}{2\pi i}\int_C \frac{\mathrm{d}y}{y} \in \mathbb{Z}.
\]
If $C$ is an arbitrary path in $\Omega^c$, by \eqref{eq6.1}, for some $n\geq 1$, $H^n(C)\in V^+_{R,c}$. Therefore, 
\[
\alpha\left(H^n(C)\right)=\frac{1}{2\pi i}\int_{H^n(C)}\omega = \frac{1}{2\pi i}\int_C (H^n)^*(\omega)=d^n\frac{1}{2\pi i}\int_C \omega= d^n\alpha(C).
\]
Since $\alpha(H^n(C))\in \mathbb{Z}$, it follows that $\alpha(C)\in \mathbb{Z}[1/d].$

\medskip 
\noindent
Now our aim is to show that $\alpha:\pi_1(\Omega^c)\to \mathbb{Z}[1/d]$ is a group isomorphism. To establish the surjectivity, note that for sufficiently large $R'$ and for $m \in \mathbb{Z}$,
\[
C_m : t \mapsto \bigl(0, R'e^{2\pi i m t}\bigr)
\]
is a path in $V^+_{R,c}$ such that $\alpha(C_m) = m$. Moreover, for $n \in \mathbb{N}$, $(H^n)^{-1}(C_m)$ is a closed path in $\Omega^c$ that satisfies
\[
\alpha\bigl((H^n)^{-1}(C_m)\bigr) = \frac{m}{d^n}.
\]
\noindent
To prove the injectivity, suppose $C$ is a closed path in $\Omega^c$ such that $\alpha(C)=0$. This  implies $\alpha(H^n(C))=0$,  for all $n\geq 0$. We show that $C$ is null-homotopic in $\Omega^c$. Let $c_2>c_1>c$ be such that 
\[
C\subset \left \{(x,y)\in \mathbb{C}^2:c_1<G_H^+(x,y)<c_2\right \}.
\]
For $\epsilon>0$, let $n_\epsilon$ be such that $d^{n_\epsilon}(c_1-c)>2\epsilon$ and $R_\epsilon>R$ be such that $H^{n_\epsilon}(C)\subset V^+_{R_\epsilon}$. Let
\[
R_1=\inf\left\{|y|: \text{ there exists $x\in \mathbb{C}$ such that } (x,y)\in H^{n_\epsilon}(C)\right\},
\]
and 
\[
R_2=\sup\left\{|y|: \text{ there exists $x\in \mathbb{C}$ such that } (x,y)\in H^{n_\epsilon}(C)\right\}.
\]
For $(x,y)\in V^+_{R_\epsilon}\cap \{(x,y)\in \mathbb{C}^2 : R_1\leq |y| \leq R_2\}$, there exists $(x_1,y_1)\in H^{n_\epsilon}(C)$ such that $|y|=|y_1|$. Also, note that
\[
G_H^+(x,y)\geq \log^+|y|-\epsilon=\log^+|y_1|-\epsilon\geq G_H^+(x_1,y_1)-2\epsilon\geq d^{n_\epsilon}c_1-2\epsilon > d^{n_\epsilon}c.
\]
Thus $(x,y)\in \Omega^{d^{n_\epsilon}c}$, i.e., $V^+_{R_\epsilon}\cap \{(x,y)\in \mathbb{C}^2 : R_1\leq |y| \leq R_2\} \subset \Omega^{d^{n_\epsilon}c}.$ Now since $\alpha(H^{n_\epsilon}(C))=0$, it follows that $H^{n_\epsilon}(C)$ is null homotopic in  $V^+_{R_\epsilon}\cap \{(x,y)\in \mathbb{C}^2 : R_1-\epsilon< |y| < R_2+\epsilon\}\subset \Omega^{d^{n_\epsilon}c}$. Therefore, $C$ is null homotopic in $\Omega^c$.

\medskip 
\noindent
{\it Step 3:} 
In this step we construct the intermediate covering of $\Omega^c$ corresponding to the subgroup $\mathbb{Z}\leq \mathbb{Z}[1/d]$. 
Fix $a \in V_R^+\cap \Omega^c$. Define an equivalence relation on the set 
$$ 
\left\{(z,C) : z \in \Omega^c \text{ and } C \text{ is a path from $a$ to $z$ in }\Omega^c\right\}
$$
 by
\[
(z,C)\sim (z',C') \text{ if and only if } z=z' \text{ and } [C\bar C'] \in \mathbb{Z}.
\] 
Let 
$$
\hat \Omega^c = \{(z,C): z \in \Omega^c \text{ and } C \text{ is a path from $a$ to $z$ in }\Omega^c\}/\sim
$$
be the set of equivalence classes. Define 
$
\hat \pi^c: \hat \Omega^c \to \Omega^c \text{ by } [(z,C)]\mapsto z.
$
We pull back the complex structure of $\Omega^c$ to give a complex structure to $\hat{\Omega}^c$ so that the map $\hat \pi^c$ is holomorphic. In a similar manner, for $n\geq 0$, we can construct a covering $\hat \Omega^{d^n c}$ of $\Omega^{d^n c}$ with the base point $H^n(a)$. For $n\geq 0$, define 
\[
\hat{H}_{d^n c}:\hat{\Omega}^{d^n c} \to \hat{\Omega}^{d^{n+1} c} \text{ \hspace{1mm} by \hspace{1mm}} [(z,C)] \mapsto [(H(z),H(C))].
\]
Let 
$$
\hat{H}^{n}_{c}=\hat{H}_{d^{n-1}c}\circ \cdots \circ \hat{H}_c : \hat{\Omega}^c \to\hat{\Omega}^{d^n c},
$$
then $\hat{H}^{n}_{c}([(z,C)])=[(H^n(z),H^n(C))]$.
For $n\geq 0$, define $\hat{\phi}_{d^n c}:\hat \Omega^{d^n c} \to \mathbb{C}\backslash\mathbb{\overline{D}} $ by
\[
[(z,C)]\mapsto \phi_H^+(H^n(a))\left(e^{\int_C \omega}\right).
\]
Then 
\begin{equation}\label{eq6.2}
\hat{\phi}_{d^{n+1} c}\left(\hat{H}_{d^n c}\left([(z,C)]\right)\right)=\phi_H^+\left(H^{n+1}(a)\right)\left(e^{\int_{H(C)}\omega}\right)=\left(\phi_H^+(H^n(a))\left(e^{\int_C\omega}\right)\right)^d=\left(\hat{\phi}_{d^n c}\left([(z,C)]\right)\right)^d,
\end{equation}
for $[z,C]\in \hat{\Omega}^{d^nc}$. Inductively we get,
\begin{equation}\label{eq6.3}
\hat{\phi}_{d^n c} \circ \hat{H}^{n}_{c}=\hat{\phi}_{d^n c} \circ \hat{H}_{d^{n-1}c}\circ \cdots \circ \hat{H}_c = \left(\hat{\phi}_{d^{n-1}c}(\hat{H}_{d^{n-2}c}\circ \cdots \circ \hat{H}_c)\right)^d=\cdots= \left(\hat{\phi}_c\right)^{d^n}.
\end{equation}
Let 
\[\hat{V}^+_{R,d^n c}=\{[(z,C)]\in \hat{\Omega}^{d^n c} : z\in V_{R,d^n c}^+, C \text{ is a path from $H^n(a)$ to $z$ in } V_{R,d^n c}^+\}
\]
and 
\[
\hat{U}^+_{R,d^n c}=\{[(z,C)]\in \hat{V}^+_{R,d^n c} : z \in U^+_{R,d^n c}\}.
\]
Then using \eqref{eq6.1}, it is easy to see that  
\begin{equation}\label{eq6.4}
\hat{\Omega}^c= \bigcup_{n\geq 0}(\hat{H}^{n}_{c})^{-1}(\hat{V}^+_{R,d^n c})=\bigcup_{n\geq 0}(\hat{H}^{n}_{c})^{-1}(\hat{U}^+_{R,d^n c}),
\end{equation}
where both the unions are increasing.
Recall $\psi_H^+$ defined in Lemma \ref{prop2.2}.
For $[(z,C)]\in \hat{U}^+_{R,c}$, define  $\hat \psi_c\left([(z,C)]\right)=\psi_H^+(z)$. If $[(z,C)]\in (\hat{H}_c)^{-1}\left(\hat{U}^+_{R,d c}\right)$, then \eqref{psi} gives
\[
\hat{\psi}_{d c}\left(\hat{H}_c\left([(z,C)]\right)\right)=\psi_H^+\left(H(z)\right)=\left(\frac{\delta}{d}\right)\psi_H^+(z)+Q\left(\phi_H^+(z)\right).
\]
If $[(z,C)]\in \hat{\Omega}^c$, then by \eqref{eq6.4} there exists an $n\geq 1$ such that $\hat{H}^n_c\left([(z,C)]\right)\in \hat{U}^+_{R,d^n c}$. Thus using the functorial equation \eqref{psi} satisfied by $\psi_H^+$, we get,
\[
\hat{\psi}_{d^n c}\left(\hat{H}^n_c\left([(z,C)]\right)\right)=\psi_H^+\left(H^n(z)\right)=\left(\frac{\delta}{d}\right)^n\psi_H^+(z)+\left(\frac{\delta}{d}\right)^{n-1}Q\left(\phi_H^+(z)\right)+\cdots+Q\left(\left(\phi_H^+(z)\right)^{d^{n-1}}\right).
\]
So, set
\[
\hat \psi_c\left([(z,C)]\right)=\left(\frac{d}{\delta}\right)^n\hat{\psi}_{d^n c}\left(\hat{H}^n_c\left([(z,C)]\right)\right)-\left(\frac{d}{\delta}\right)^n Q\left(\hat \phi_{d^{n-1}c}\left(\hat H^{n-1}_c\left([(z,C)]\right)\right)\right)+\cdots+\left(\frac{d}{\delta}\right)Q\left(\hat \phi_c\left([(z,C)]\right)\right).
\]
Similarly, for $n\geq 0$, we can define holomorphic functions $\hat{\psi}_{d^n c}:\hat{\Omega}^{d^n c}\to \mathbb{C}$ such that for $[(z,C)]\in \hat{\Omega}^{d^n c}$, we have
\begin{equation}\label{eq6.5}
\hat{\psi}_{d^{n+1} c}\circ \hat{H}_{d^n c}\left([(z,C)]\right)=\left(\frac{\delta}{d}\right)\hat{\psi}_{d^n c}\left([(z,C)]\right)+Q\left(\hat{\phi}_{d^n c}\left([(z,C)]\right)\right).
\end{equation}

\medskip
\noindent
{\it Step 4:}
Now we show that $\hat{\Omega}^c$ is biholomorphic to $\mathbb{C}\times \mathcal{A}^c$. First note that if $[(z,C)]\in \hat{\Omega}^c$, then by \eqref{eq6.4} there exist an $n>1$ such that $\hat{H}^n_c\left([(z,C)]\right)\in \hat{V}^+_{d^n c}$. Now
\[
G_H^+=\log\vert\phi_H^+\vert,
\] 
on $V_R^+$. Thus
\[
\log\left\vert\hat{\phi}_{d^n c}\left(\hat{H}^n_c\left([(z,C)]\right)\right)\right\vert=\log\left\vert\phi_H^+\left(H^n(z)\right)\right\vert=G_H^+\left(H^n(z)\right)=d^n G_H^+(z)>d^n c.
\]
Thus, \eqref{eq6.3} gives 
\[
d^n \log|\hat{\phi}_c([(z,C)])|>d^n c \implies \log|\hat{\phi}_c([(z,C)])|>c.
\]
Thus the image of $\hat{\Omega}^c$ under the function $\hat{\phi}_c$ is contained in $\mathcal{A}^c$.
For $n\geq 0$, define $\hat{\Phi}_{d^n c}:\hat{\Omega}^{d^n c}\to \mathbb{C}\times \mathcal{A}^{d^n c}$ by 
\[
[(z,C)]\mapsto\left(\hat{\psi}_{d^n c}\left([(z,C)]\right),\hat{\phi}_{d^n c}\left([(z,C)]\right)\right),
\]
and $G_{d^n c}:\mathbb{C}\times \mathcal{A}^{d^n c}\to \mathbb{C}\times\mathcal{A}^{d^{n+1} c}$ by 
\[
(s,t)\mapsto\left(\frac{\delta}{d}s+Q(t),t^d\right).
\]
It follows from \eqref{eq6.2} and \eqref{eq6.5}, that for $n\geq 0$, 
\[
\begin{tikzcd}
\mathbb{C}\times\mathcal{A}^{d^n c}
  \arrow[r, "G_{d^n c}"]
&
\mathbb{C}\times\mathcal{A}^{d^{n+1} c}
\\
\hat{\Omega}^{d^n c}
  \arrow[u, "\hat{\Phi}_{d^n c}"]
  \arrow[r, "\hat{H}_{d^n c}"']
&
\hat{\Omega}^{d^{n+1} c}
  \arrow[u, "\hat{\Phi}_{d^{n+1} c}"']
\end{tikzcd}
\]
the above diagram commutes, i.e.,
\begin{equation}\label{eq6.6}
\hat{\Phi}_{d^{n+1}c}\circ \hat{H}_{d^n c}= G_{d^n c}\circ \hat{\Phi}_{d^n c}.
\end{equation}
Inductively, for $n\geq 0$, we get
\begin{equation}\label{eq6.7}
\hat{\Phi}_{d^{n+1}c}\circ \hat{H}_c^{n+1}=\hat{\Phi}_{d^{n+1}c}\circ \hat{H}_{d^n c}\circ \cdots \circ \hat{H}_c =G_{d^n c}\circ\cdots\circ G_c\circ\hat{\Phi}_c.
\end{equation}
For $n\geq 0$, let $W_{d^n c}=\hat{\Phi}_{d^n c}\left(\hat{U}^+_{R,d^n c}\right).$ 
We claim that 
\[
\bigcup_{n=0}^{\infty}\left(G_{d^{n-1}c}\circ\cdots\circ G_c\right)^{-1}\left(W_{d^n c}\right)=\mathbb{C}\times\mathcal{A}^c,
\] 
where, by convention, when $n=0$, $\left(G_{d^{n-1}c}\circ\cdots\circ G_c\right)^{-1}\left(W_{d^n c}\right)=W_c$.
Since $\hat{\Phi}_{dc}\circ \hat{H}_{c}= G_{c}\circ \hat{\Phi}_{c}$ and $\hat{H}_c\left(\hat{U}^+_{R,c}\right)\subset\hat{U}^+_{R,dc}$, $G_c\left(W_c\right)\subset W_{dc}$. Similarly, it follows from \eqref{eq6.4} and \eqref{eq6.6} that the above union is increasing.
To see the claim, let $(s,t)\in \mathbb{C}\times\mathcal{A}^c$, then 
\[
\left(G_{d^{n-1}c}\circ\cdots\circ G_c\right)(s,t)=\left(\left(\frac{\delta}{d}\right)^{n}s+\left(\frac{\delta}{d}\right)^{n-1}Q(t)+\left(\frac{\delta}{d}\right)^{n-2}Q(t^d)+\cdots+Q(t^{d^{n-1}}),t^{d^n}\right).
\]
Since by Lemma \ref{prop2.2}, $\{(s,t):|t|>\rho_1,|s|<\rho_2|t|^2\}\subset\Phi_H^+(U_R^+),$ for large $n\geq 1$ 
\[
\left(G_{d^{n-1}c}\circ\cdots\circ G_c\right)(s,t)\in \Phi_H^+(U_R^+).
\]
Let $z\in U_R^+$ such that $\Phi_H^+(z)=\left(\psi_H^+(z),\phi_H^+(z)\right)=\left(G_{d^{n-1}c}\circ\cdots\circ G_c\right)(s,t)$, then 
\[
\vert\phi_H^+(z)\vert=\vert t\vert^{d^n}>e^{d^n c}\implies G_H^+(z)>d^n c.
\] 
Thus, for some $n$, $G_{d^{n-1}c}\circ\cdots\circ G_c(s,t)\in \Phi_H^+(U^+_{R,d^n c})=\hat{\Phi}_{d^n c}\left(\hat{U}^+_{d^n c}\right)=W_{d^n c}$.  The other inclusion is trivial.
Now we will show that 
\begin{equation}\label{eq6.8}
\hat{\Phi}_c:\hat{\Omega}_c=\bigcup_{n\geq 0}\left(\hat{H}_c^n\right)^{-1}\left(\hat{U}^+_{R,d^n c}\right)\to \mathbb{C}\times\mathcal{A}^c=\bigcup_{n\geq 0}\left(G_{d^{n-1}c}\circ\cdots\circ G_c\right)^{-1}\left(W_{d^n c}\right)
\end{equation}
is a biholomorphism. Note that both the unions in \eqref{eq6.8} are increasing. Let $[(z,C)]\in \left(\hat{H}_c^n\right)^{-1}\left(\hat{U}^+_{R,d^n c}\right)$, then 
\begin{align*}
\hat{\Phi}_c\left([(z,C)]\right)\in & \left(G_{d^{n-1}c}\circ\cdots\circ G_c\right)^{-1}\left(\left(\hat{\Phi}_{d^n c}\circ\hat{H}^n_c\right)\left(\left(\hat{H}_c^n\right)^{-1}\left(\hat{U}^+_{R,d^n c}\right)\right)\right)\\
& =\left(G_{d^{n-1}c}\circ\cdots\circ G_c\right)^{-1}\left(W_{d^n c}\right).
\end{align*}
Thus, it is enough to show that for every $n\geq 0$, 
\[
\hat{\Phi}_c:\left(\hat{H}_c^n\right)^{-1}\left(\hat{U}^+_{R,d^n c}\right)\to \left(G_{d^{n-1}c}\circ\cdots\circ G_c\right)^{-1}\left(W_{d^n c}\right)
\]
is a bijection. To prove the injectivity, let $[(z,C)],[(z',C')]\in \left(\hat{H}_c^n\right)^{-1}\left(\hat{U}^+_{R,d^n c}\right)$ be such that 
\[
\hat{\Phi}_c\left([(z,C)]\right)=\hat{\Phi}_c\left([(z',C')]\right)\implies \left(\hat{\Phi}_{d^{n+1}c}\circ \hat{H}^n_c\right)\left([(z,C)]\right)=\left(\hat{\Phi}_{d^{n+1}c}\circ \hat{H}^n_c\right)\left([(z',C')]\right),
\]
which in turn gives 
\begin{align*}
    \Phi_H^+(H^n(z))=\Phi_H^+(H^n(z'))\implies H^n(z)=H^n(z')\implies z=z'.
\end{align*}
Now $\hat{\phi}_c([(z,C)])=\hat{\phi}_c([(z',C')])$ gives $\int_{C\overline{C'}}\omega=0$. Hence $[(z,C)]=[(z',C)]$. Now we prove surjectivity. Let $(s,t)\in W_{d^n c}$, then there exists $[(z,C)]\in \hat U^+_{R,d^n c}$ such that $\hat{\Phi}_{d^n c}\left([(z,C)]\right)=(s,t)$. Now \eqref{eq6.7} gives that $\hat{\Phi}_c$ maps elements of $\{(\hat{H}_c^n)^{-1}\left([(z,C)]\right)\}$ to elements of  $\{\left(G_{d^{n-1}c}\circ\cdots\circ G_c\right)^{-1}(s,t)\}$. Since both the sets has $d^n$ elements, thus injectivity implies surjectivity.

\medskip
\noindent
Further, it follows from \eqref{eq6.2} and \eqref{eq6.5}, that for $n\geq 0$, the lift $\tilde{H}_{d^n c}:\mathbb{C}\times\mathcal{A}^{d^n c} \to \mathbb{C}\times\mathcal{A}^{d^{n+1} c}$ of $H:\Omega^{d^n c}\to\Omega^{d^{n+1} c}$ is given by 
\[
\tilde{H}_{d^n c}(s,t)=\left(\frac{\delta}{d}s+Q(t),t^d\right).
\]
 \medskip
 \noindent
{\it Step 5:}
In this step, we show that for every $\frac{k}{d^n}\in \mathbb{Z}[1/d]/\mathbb{Z}$, there exists a unique deck transformation $\gamma_{c,\frac{k}{d^n}}:\mathbb{C}\times\mathcal{A}_c\to\mathbb{C}\times\mathcal{A}_c$ defined by
\begin{equation}\label{deck}
\gamma_{c,\frac{k}{d^n}}(s,t)=\left(s+\frac{d}{\delta}\sum_{l=0}^{\infty}\left(\frac{d}{\delta}\right)^l\left(Q(t^{d^l})-Q\left((e^{2\pi i k/d^n}t)^{d^l}\right)\right),\,e^{2\pi i k/d^n}t\right).
\end{equation}
Since the expression for $\gamma_{c,\frac{k}{d^n}}$ is independent of $c$, we suppress the dependence on $c$ and simply write
\[
\gamma_{c,\frac{k}{d^n}}=\gamma_{\frac{k}{d^n}}.
\]
First we prove that for every $k\geq1,$ $n\ge  0$ the following diagram commute,
\[
\begin{tikzcd}
\mathbb{C}\times\mathcal{A}^{c}
  \arrow[d, "\gamma_{\frac{k}{d^{n+1}}}"']
  \arrow[r, "\hat{H}_c"]
&
\mathbb{C}\times\mathcal{A}^{d c}
  \arrow[d, "\gamma_{\frac{k}{d^{n}}}"]
\\
\mathbb{C}\times\mathcal{A}^{c}
  \arrow[r, "\hat{H}_{c}"']
&
\mathbb{C}\times\mathcal{A}^{d c}
\end{tikzcd}
\]
that is, $\tilde{H}_c\circ \gamma_{\frac{k}{d^{n+1}}}=\gamma_{\frac{k}{d^n}}\circ\tilde{H}_c$.
Note that $\hat{\pi}^c\circ\tilde{H}_c\circ\gamma_{\frac{k}{d^{n+1}}}=H\circ\hat{\pi}^c$. Thus both $\tilde{H}_c\circ\gamma_{\frac{k}{d^{n+1}}}$ and $\tilde{H}_c$ are lifts of $H$. Since two lifts of $H$ differ by a unique deck transformation. Thus $\tilde{H}_c\circ\gamma_{\frac{k}{d^{n+1}}}=\mathcal{\Phi}\left(\gamma_{\frac{k}{d^{n+1}}}\right)\circ\tilde{H}_c$. Clearly, $\mathcal{\Phi}$ is a homomorphism on the group of deck transformations. Since $H$ induces multiplication by $d$ on the fundamental group of $\Omega^c$, $\mathcal{\Phi}\left(\gamma_{\frac{k}{d^{n+1}}}\right)=\gamma_{\frac{k}{d^{n}}}$. Thus we have commutativity of above diagram.
Now the relation $\tilde{H}_c\circ \gamma_{\frac{1}{d^{n+1}}}=\gamma_{\frac{1}{d^n}}\circ\tilde{H}_c$ allows us to determine $\gamma_{1/d^n}$ inductively. For $n=0$, the map $\gamma_1$ is the identity. Assume that, for some $n\geq 0$, the deck transformation $\gamma_{1/d^n}$ is given by \eqref{deck}, and let $\gamma_{1/d^{n+1}}(s,t)=(s_1,t_1).$
Then
\[
\left(\frac{\delta}{d}s_1+Q(t_1),\,t_1^d\right)
=
\left(
\frac{\delta}{d}s+Q(t)
+\frac{\delta}{d}\sum_{l=0}^{\infty}\left(\frac{d}{\delta}\right)^l
\left(Q\left(t^{d^{l+1}}\right)-
Q\left(\left(e^{2\pi i/d^{n+1}}t\right)^{d^{l+1}}\right)\right),
\,e^{2\pi i/d^n}t^d
\right).
\]
Hence,
\begin{align*}
t_1^d
&=
\left(e^{2\pi i/d^{n+1}}t\right)^d,\\
s_1
&=
s+\frac{d}{\delta}\bigl(Q(t)-Q(t_1)\bigr)
+\frac{d}{\delta}\sum_{l=0}^{\infty}\left(\frac{d}{\delta}\right)^{l+1}
\left(
Q(t^{d^{l+1}})
-
Q\!\left((e^{2\pi i/d^{n+1}}t)^{d^{l+1}}\right)
\right).
\end{align*}
Comparing above equations, we obtain precisely the formula in \eqref{deck}. This completes the induction for the case $k=1$. Finally, using the identity $\gamma_{\frac{k+1}{d^n}}=\gamma_{\frac{k}{d^n}}\circ\gamma_{\frac{1}{d^n}},$ inductively we  get $\gamma_{k/d^n}$ for all ${k}/{d^n}\in\mathbb{Z}[1/d]/\mathbb{Z}$, proving the claim.

%In \cite{HOV}, Hubbard and Obereste-Vorth gave a neat and useful description of intermediate covering $\tilde{U}_H^+$ of $U_H^+$ corresponding to the subgroup $\mathbb{Z}\subseteq \mathbb{Z}[1/d]$. It was shown that $\tilde{U}_H \cong \mathbb{C}\times (\mathbb{C}\setminus{\bar{\mathbb{D}}})$

\section{Lifts of automorphisms of  $\Omega^c$ and $U_H^+$}
In this section, we study the form of lifts of automorphisms of $\Omega^c$ and $U_H^+$ using an idea of Bousch in \cite{Bousch}. Also see \cite[Section~6]{BPV}, where automorphisms of punctured Short $\mathbb{C}^2$ are studied. Recall from Section 2 and Section 3, respectively, that the covering of $U_H^+$ corresponding to the subgroup $\mathbb{Z} \leq \mathbb{Z}[1/d]$ is $\mathbb{C} \times \mathbb{C}\backslash\mathbb{\overline{D}}$, and the covering of $\Omega^c$ corresponding to the subgroup $\mathbb{Z} \leq \mathbb{Z}[1/d]$ is $\mathbb{C} \times \mathcal{A}^c$, with $\mathcal{A}^c=\{t\in\mathbb{C}:\vert t\vert >e^c\}$. 

\begin{thm}\label{lift of UH+}
(i) Let $f:\Omega^c \to \Omega^c$ be an automorphism which lifts to an automorphism on the covering space $\mathbb{C}\times \mathcal{A}^c$. Then there exists a lift $\tilde{f}:\mathbb{C}\times\mathcal{A}^c \to \mathbb{C}\times\mathcal{A}^c$ of $f$ of the form $\tilde{f}(s,t)=(\beta s +\gamma,\alpha t)$ with $\beta^{d-1}=1$, $\beta = \alpha^{d+1}$, and $\gamma \in \mathbb{C}$.

\medskip 
\noindent 
(ii) Let $f:U_H^+ \to U_H^+$ be an automorphism which lifts to an automorphism on the covering space $\mathbb{C}\times\mathbb{C}\backslash\mathbb{\overline{D}}$. Then there exists a lift $\tilde{f}:\mathbb{C}\times\mathbb{C}\backslash\mathbb{\overline{D}} \to \mathbb{C}\times\mathbb{C}\backslash\mathbb{\overline{D}}$ of $f$ of the form $\tilde{f}(s,t)=(\beta s +\gamma,\alpha t)$ with $\beta^{d-1}=1$, $\beta = \alpha^{d+1}$, and $\gamma \in \mathbb{C}$.

\end{thm}
\begin{proof}
(i) Let $\tilde{f}$ be of the form  
\[
\tilde{f}(s,t)=(\tilde{f}_1(s,t),\tilde{f}_2(s,t)).
\]
Fix $t\in \mathcal{A}^{c}$, then 
\[
s\mapsto \tilde{f}_2(s,t)
\]
is an entire function with image in $\mathcal{A}^{c}$, hence constant, i.e., $\tilde{f}_2(s,t)=\tilde{f}_2(t).$ For each fixed $t$, the map 
\[
s\mapsto\tilde{f_1}(s,t)
\]
is an automorphism of $\mathbb{C}$, and hence is of the form $\tilde{f}_1(s,t)=\beta(t)s+\gamma(t)$, with $\beta,\gamma:\mathcal{A}^{c}\to\mathbb{C}$ holomorphic functions. Let $\tilde{f}^*(s,t)=(\tilde{f}_1^*(s,t),\tilde{f}_2^*(t))$ be the inverse of $\tilde{f}$, then 
\[
\tilde{f}_2\circ\tilde{f}_2^*=Id|_{\mathcal{A}^{c}} \text{\hspace{1mm} and } \tilde{f}_2^*\circ\tilde{f}_2=Id|_{\mathcal{A}^{c}}.
\]
Thus $\tilde{f}_2:\mathcal{A}^{c}\to\mathcal{A}^{c}$ is a biholomorphism. Consider the function $h_c:\mathbb{D}^*\to \mathcal{A}^{c}$ defined by 
$h_c(t)=e^{c}/t.$
Then
$h:\mathbb{D}^*\to \mathbb{D}^*$ defined by $h(t)=h_c^{-1}\circ\tilde{f}_2\circ h_c(t)$ is a biholomorphism. Since $\tilde{f}_2:\mathcal{A}^c\to \mathcal{A}^c$ is a biholomorphism, it cannot have an essential singularity at infinity. If possible, let $\tilde{f}_2$ has a removable singularity at infinity and let $\lim_{t\to \infty}\tilde{f}_2(t)=t_0$, with $\vert t_0\vert=e^c$. Define $\tilde{g}_2:\{t\in \mathbb{C}:\vert t\vert<e^{-c}\}\to \{t\in \mathbb{C}:\vert t\vert <e^{-c}\}$ by $\tilde{g}_2(t)=1/\tilde{f}_2(1/t)$. Then $\left\vert \tilde{g}_2(0)\right\vert =\left\vert t_0^{-1}\right\vert =e^{-c}$, so by maximum modulus function, $\tilde{g}_2$ is a constant function, which is a contradiction. Thus, $\tilde{f}_2$ has a pole at $\infty$.
Thus, if we define $h(0)=0$, then $h:\mathbb{D}\to\mathbb{D}$ is an automorphism. Thus, $h(t)=\alpha' t$ with $|\alpha'|=1$ and $\tilde{f}_2(t)=\alpha t$ with $|\alpha|=1.$
Now for any $k\geq 1$, $n\geq 0$, $\tilde{f}\circ\gamma_{\frac{k}{d^n}}\circ \tilde{f}^{-1}$ is a deck transformation for the cover $\hat{\pi}^{c}:\hat \Omega^{c}\to \Omega^{c}$. Since 
\[
\left(\tilde{f}\circ\gamma_{\frac{k}{d^n}}\circ \tilde{f}^{-1}(s,t)\right)_2=e^{2\pi i k/d^n}t,
\]
we have 
$$
\tilde{f}\circ\gamma_{\frac{k}{d^n}}\circ \tilde{f}^{-1}=\gamma_{\frac{k}{d^n}} \text{, i.e., } \tilde{f}\circ\gamma_{\frac{k}{d^n}}=\gamma_{\frac{k}{d^n}}\circ \tilde{f}.
$$ 
%Thus, $\tilde{f}\circ\gamma_{\frac{k}{d^n}}=\gamma_{\frac{k}{d^n}}\circ \tilde{f}$.
 Therefore for any $(s,t)\in \mathbb{C}\times\mathcal{A}^{c}$, we have 
\begin{equation}\label{eq07.1}
    \begin{split}
        \beta\left(e^{2\pi i k/d^n}t\right)& \left(s+\frac{d}{\delta}\sum_{l=0}^{\infty}\left(\frac{d}{\delta}\right)^l\left(Q\left(t^{d^l}\right)-Q\left(\left(e^{2\pi i k/d^n}t\right)^{d^l}\right)\right)\right)+\gamma\left(e^{2\pi i k/d^n}t\right)\\
& =\beta(t)s+\gamma(t)+\frac{d}{\delta}\sum_{l=0}^{\infty}\left(\frac{d}{\delta}\right)^l\left(Q\left(\left(\alpha t\right)^{d^l}\right)-Q\left(\left(e^{2\pi i k/d^n}\alpha t\right)^{d^l}\right)\right).
    \end{split}
\end{equation}
Comparing both sides of \eqref{eq07.1}, we get $\beta(t)=\beta(e^{2\pi i k/d^n} t)$, for all $k,n\geq 1$, which in turn gives $\beta$ is a constant. Hence \eqref{eq07.1} becomes
\begin{equation}\label{eq07.2}
\begin{split}
    \frac{d}{\delta}\sum_{l=0}^{\infty}\left(\frac{d}{\delta}\right)^l\left[\left(\beta Q\left(t^{d^l}\right)-Q\left(\left(\alpha t\right)^{d^l}\right)\right)-\left(\beta Q\left(\left(e^{2\pi ik/d^n} t\right)^{d^l}\right)-Q \right. \right. & \left. \left. \left(\left(e^{2\pi ik/d^n}\alpha t\right)^{d^l}\right)\right)\right]\\
    & =\gamma(t)-\gamma(e^{2\pi i k/d^n}t).
    \end{split}
\end{equation} 
The highest degree term on L.H.S. of \eqref{eq07.2} is 
\[
\left(1-e^{2\pi i k/d}\right)\left(\beta-\alpha^{d^{n-1}(d+1)}\right)t^{d^{n-1}(d+1)}.
\]
Let $k=1$ and fix a $t$. Then R.H.S. of \eqref{eq07.2} is uniformly bounded for $n\geq 1$. Thus, 
\begin{equation}\label{alpha1}
\alpha^{d^{n-1}(d+1)}\to \beta,
\end{equation}
as $n\to \infty$. Also
\begin{equation}\label{alpha2}
\alpha^{d^n(d+1)} \to \beta,
\end{equation}
as $n\to \infty$. Thus, we get
\begin{equation}\label{alpha3}
\alpha^{d^{n-1}(d+1)(d-1)}\to 1,
\end{equation}
as $n\to \infty$. By \eqref{alpha1}, the sequence in \eqref{alpha3} converges to $\beta^{d-1}$. Thus, $\beta^{d-1}=1$. Therefore, $\beta$ is a repelling fixed point of $z\mapsto z^d$. Thus, the sequence $\{(\alpha^{(d+1)})^{d^n}\}_{n\geq 0}$ converges to $\beta$ if and only if it is an eventually constant sequence, i.e., for some $n_0\geq 1$, $\alpha^{d^{n}(d+1)}=\beta$ for all $n\geq n_0$.
Let $1\leq k_1 \leq d^2-1$ be such that
\[
\alpha^{d^{n_0}}=e^{2\pi i k_1/(d^2-1)} \implies \beta=e^{2\pi i k_1/(d-1)}.
\]
Let
\[
\tilde{\alpha}=e^{2\pi i k_1/(d^2-1)}=\alpha^{d^{n_0}}.
\]
Let $1\leq k_2 \leq d^{n_0}(d^2-1)$ be such that
\[
\alpha=e^{2\pi i k_2/d^{n_0}(d^2-1)}.
\]
Thus
\[
\tilde{\alpha}=e^{2\pi i k_2/(d^2-1)}.
\]
Therefore,
\[
\frac{\alpha}{\tilde{\alpha}}=e^{\frac{2\pi i k_2}{d^2-1}\left(\frac{1}{d^{n_0}}-1\right)}.
\]
Since we can choose $n_0$ to be even, $(d^2-1)$ divides $(1-d^{n_0})$. Hence, $\alpha=e^{{2\pi i k_3}/{d^{n_0}}}\tilde{\alpha}$, for some $k_3\geq 1$.
Let 
\[
\tilde{F}=\tilde{f}\circ \gamma^{-1}_{\frac{k_3}{d^{n_0}}}:\mathbb{C}\times \mathcal{A}^{c}\to \mathbb{C}\times \mathcal{A}^{c}.
\] 
Then $\tilde{F}$ is a lift of $f$ and 
\[
F:(s,t)\mapsto(\beta s+\gamma(t),\tilde{\alpha}t),
\]
with $\beta^{d-1}=1$, $\tilde{\alpha}^{d+1}=\beta$. By abuse of notation, we still write $\tilde{F}=\tilde{f}$ and $\tilde{\alpha}=\alpha$.
Since $Q(t)=t^{d+1}+b_{d-1}t^{d-1}+\cdots+b_0$, \eqref{eq07.2} becomes
\begin{equation}
\begin{split}
\frac{d}{\delta}\sum_{l=0}^{n-1}& \left(\frac{d}{\delta}\right)^l\Bigg[\left(1-e^{2\pi i(d+1)/d^{n-l}}\right)\left(\beta-\alpha^{d^l(d+1)}\right)t^{d^l(d+1)}\Bigg.\\
& +b_{d-1}\left(1-e^{2\pi i(d-1)/d^{n-l}}\right)\left(\beta-\alpha^{d^l(d-1)}\right)t^{d^l(d-1)}+\\
& \Bigg. \cdots+b_1\left(1-e^{2\pi i/d^{n-l}}\right)\left(\beta-\alpha^{d^l}\right)t^{d^l}\Bigg] =\gamma(t)-\gamma(e^{2\pi i k/d^n}t).
\end{split}
\end{equation}
We claim that the L.H.S. of the above equation is zero. Since $\alpha^{d+1} = \beta$, the coefficients of $t^{d^l(d+1)}$ are zero, for $0 \le l \le n-1$.
If for some $1\leq i\leq d-1$, $b_{d-i}\neq 0$, then 
\[
\alpha^{d^{n-1}(d-i)}\to \beta
\]
as $n\to \infty$. Now since $\beta$ is a repelling fixed point of $z\mapsto z^d$, there exist an integer $n_1$ such that $\alpha^{d^n(d-i)}=\beta$ for all $n\geq n_1$. Since $\alpha^{d^2}=\alpha$, we have $\alpha^{d^{2n_1}}=\alpha$. Thus, $\alpha^{d-i}=\beta$ which implies $\alpha^{d^n(d-i)}=\beta$ for $n\geq 0$, thus the claim follows. This implies 
\[
\gamma(t)-\gamma(e^{2\pi i k/d^n}t)=0, \text{ for all } k.n\geq 1 
\]
Thus $\gamma \equiv {\rm{const}}$.
Any other lift of $f$ is of the form $\tilde{f}\circ\gamma_{\frac{k}{d^n}}$.  Therefore it follows that corresponding to any automorphism $f$ of $\Omega^c$, there exists a lift $F:(s,t)=(\beta s+\gamma, \alpha t)$ from $\mathbb{C}\times \mathcal{A}^c$ to itself with $\beta^{d-1}=1$ and $\alpha^{d+1}=\beta$.

\medskip 
\no 
(ii) One can run a similar proof as above for this part. Therefore, we skip the details. 
\end{proof}

\begin{rem}
By the work in \cite{Pal}, it follows that a lift $F:\mathbb{C}\times\mathbb{C}\backslash\mathbb{\overline{D}}\to \mathbb{C}\times\mathbb{C}\backslash\mathbb{\overline{D}}$ of any $f\in \mathrm{Aut}(U_H^+)$, which induces identity on the fundamental group of $U_H^+$, is of the form $(s,t)\mapsto (\beta s+\gamma(t),\alpha t)$. Here $\beta^{d-1}=1$ and there exists an $n_0\geq 1$ such that $\alpha^{d^n(d+1)}=\beta$, for all $n\geq n_0$. It follows from the second part of Theorem \ref{lift of UH+} that there exists a lift of $f$ of the form $(s,t)\mapsto (\beta s+\gamma,\alpha t)$ with $\beta^{d-1}=1$ and $\alpha^{d+1}=\beta$. 

%Similar to what we have done in the case of lifts of certain automorphisms of $\Omega^c$, we can compose $F$ with some appropriate deck transformation such that $\alpha^{d^n(d+1)}=\beta$ for all $n\geq 0$ and $\gamma$ is constant.  
\end{rem}
\end{proof}
\section{Proofs of Theorem \ref{Thm2} and Theorem \ref{Thm1} }
\noindent
\subsection*{Proof of Theorem \ref{Thm2}:}
{\it Step 1:} 
In this step, we show that $f$ preserves level sets of the Green's function. For $c>0$, let $\Omega_c=\{z\in \mathbb{C}^2 : G_H^+(z)<c\}$, and let $\Omega_c'=\{z\in \mathbb{C}^2 : 0<G_H^+(z)<c\}$. %Then $\partial\Omega_c =\{ z \in \mathbb{C} : G_H^+(z)=c \}$. 
Let \[\pi_H:\mathbb{C}\times \mathbb{C}\backslash\mathbb{\overline{D}} \to U_H^+\] be the covering map.
Let $\tilde{f}$ be a lift of $f$ to $\mathbb{C}\times \mathbb{C}\backslash \mathbb{\overline{D}}$, which is of the form 
\[
\tilde{f}(s,t)=(\beta s +\gamma , \alpha t),
\]
with constant $\alpha,\beta,\gamma\in \mathbb{C}$ such that $\alpha^{d+1}=\beta$, $\beta^{d-1}=1$. 
Clearly, $\tilde{f}(\mathbb{C}\times \mathcal{A}_c)=\mathbb{C}\times \mathcal{A}_c$, where $\mathcal{A}_c=\{t\in \mathbb{C} : 1<\vert t\vert <e^c\}$. Then 
\begin{equation} \label{fOmegac}
f\circ \pi_H(\mathbb{C}\times \mathcal{A}_c)=\pi_H\circ \tilde{f}(\mathbb{C}\times \mathcal{A}_c)=\pi_H(\mathbb{C}\times \mathcal{A}_c). 
\end{equation}
Now $\pi_H(\mathbb{C}\times \mathcal{A}_c)=\Omega_c'$.
It follows from the construction of $\pi_H$, in particular from the fact that the level sets of B\"{o}ttcher coordinate $\hat{\phi}_H^+: \hat{U}_H^+ \rightarrow \mathbb{C}\times \mathbb{C}\backslash\mathbb{\overline{D}}$ straightens up when it is lifted to $\mathbb{C}\times \mathbb{C}\backslash\mathbb{\overline{D}}$. 
%Since for every $c>0$, $\pi_H(\mathbb{C}\times\mathcal{A}^c)=\Omega^c$, $\pi_H(\mathbb{C}\times\{t\in \mathbb{C}: \vert t\vert =e^c\})=\{z\in \mathbb{C}^2:G_H^+(z)=c\}$.
Thus, from (\ref{fOmegac}), it follows that  $f(\Omega_c')=\Omega_c'$.
Therefore, $f$ acts as an automorphism of $\Omega_c^{'}$. 

\medskip
\noindent
The Green's function $G_H^+$ is pluriharmonic on $U_H^+$.  Thus $G_H^+\circ f$ is pluriharmonic on $U_H^+$. Now the leaves of the foliation of $\partial\Omega_c$ are biholomorphic to $\mathbb{C}$ and dense in $\partial\Omega_c$. Further, $G_H^+\circ f$ is bounded on each leaf since $f(\Omega_c')=\Omega_c'$.  Hence $f$ is identically constant on $\partial\Omega_c$. So, there exists $b>0$ such that $f(\partial\Omega_c)\subset\partial\Omega_b$. Since $f^{-1}$ also induces identity on $\pi_1(U_H^+)$,  there exists $c'>0$ such that $f^{-1}(\partial\Omega_b)\subset \partial\Omega_{c'}$. Let $z\in \partial\Omega_c$, then for some $w\in \partial\Omega_b$, $f^{-1}(w)=z$. Thus $G_H^+\circ f^{-1}(w)=G_H^+(z)=c$, which implies $  f^{-1}(\partial\Omega_b)\subset \partial\Omega_c$. Hence, $f(\partial\Omega_c)=\partial\Omega_b$. Further, since $f|_{\Omega_c^{'}}$ is an automorphism, it follows that $b=c$, i.e., $f(\partial\Omega_c)=\partial\Omega_c$. Thus, 
\[G_H^+\circ f \equiv G_H^+\]
 on $U_H^+$.

 \medskip
 \noindent
{\it Step 2:} 
In this step, we construct an open set $\mathcal{B}_1\subset U_R^+$ such that $f\left(\mathcal{B}_1\right)\subset U_R^+$. Recall from Section~2, for large $M>0$, $U_R^+=\left\{(x,y)\in V_R^+ : \left\lvert \phi_H^+(x,y)\right\rvert>M\max\{R,|x|\}\right\}.$ It follows from \cite[Proposition 5.4]{HOV}, that on $U_R^+\subset V_R^+$, the B\"{o}ttcher function
\[
\phi_H^+(x,y)=u_0(y)+u_1(y)x+\cdots,
\] 
where \[
u_0(y)=y+O\left(\frac{1}{|y|}\right), \text{ and } u_1(y)=-\frac{\delta}{dy^{d-1}}+o\left(\frac{1}{|y|^{d-1}}\right).
\]
Since the map
$(x,t)\mapsto (x,\lambda_H^+(x,t))$
is the inverse of $(x,y)\mapsto (x,t)=(x,\phi_H^+(x,y))$, it follows that
\[
\lambda_H^+(x,t)=t+O\left(\frac{1}{|t|}\right)+\left(\frac{\delta}{d t^{d-1}}+o\left(\frac{1}{|t|^{d-1}}\right)\right)x+\cdots.
\]
By \eqref{psi_H^+},
\[
\psi_H^+(x,y)=\phi_H^+(x,y) \int_0^x \frac{\partial\lambda_H^+}{\partial t}(\xi,t)\mathrm{d}\xi +r(t),
\]
where
\[r(t)=\frac{d}{\delta}\left(\frac{c_1}{t}+\frac{c_2}{t^2}+\cdots\right)+\left(\frac{d}{\delta}\right)^2\left(\frac{c_1}{t^d}+\frac{c_2}{t^{2d}}+\cdots\right)+\cdots.
\]
Thus, for a sufficiently large filtration radius $R>1$, we have
\begin{equation}\label{eq1}
     \psi_H^+(x,y)\sim x y+\kappa \text{ and } \phi_H^+(x,y)\sim y,
\end{equation}
on $U_R^+$, where $\kappa\in \mathbb{C}$ is a constant. Also $\lambda_H^+(x,t)\sim t$ on $\left\{(x,t)\in \mathbb{C}^2:\max\{M|x|,RM\}<|t|\right\}.$
Let 
\[
\mathcal{B}=\left(\Phi_H^+\right)^{-1}\left(\left\{(s,t)\in \mathbb{C}^2:|t|>\rho_1, |s|<\rho_2|t|^2-|\gamma|\right\}\right)\cap U_R^+.
\]
For $(x,y) \in \mathcal{B}$, there exists $(s_0,t_0)=\Phi_H^+(x,y)\in \left(\pi_H\circ \left(\hat{\Phi}_H^+\right)^{-1}\right)^{-1}(x,y)$ such that 
\[
s_0 \approx x y+\kappa \text{ and } t_0\approx y.
\]
Now $\tilde{f}(s_0,t_0)=(\beta s_0 +\gamma,\alpha t_0)\in \mathbb{C}\times \mathbb{C}\backslash\mathbb{\overline{D}}$. Further
\[
|\beta s_0+\gamma|\leq |s_0|+|\gamma|<\rho_2|t|^2-|\gamma|-|\gamma|=\rho_2|t|^2,
\]
and
\[
|\alpha t_0|=|t_0|>\rho_1,
\]
i.e., $\tilde{f}(s_0,t_0)\in \Phi_H^+(U_R^+)$. Thus,  
\begin{align*}
    \pi_H\circ & (\hat{\Phi}_H^+)^{-1}\circ \tilde{f} (s_0,t_0)\in \pi_H\circ (\hat{\Phi}_H^+)^{-1}\circ \Phi_H^+(U_R^+),
    \end{align*}
    which implies
    \begin{align*}
     f\circ \pi_H\circ (\hat{\Phi}_H^+)^{-1}(s_0,t_0) \in \pi_H\circ (\hat{\Phi}_H^+)^{-1}\circ \hat{\Phi}_H^+(\hat U_R^+).
\end{align*}
\noindent
Therefore,  $f(x,y) \in \pi_H(\hat U_R^+)=U^+_R$. Hence $f(\mathcal{B})\subset U_R^+$.

\medskip
\noindent
We choose $R>0$ large such that
\[
 (1-\epsilon_0)|y|<|\phi_H^+(x,y)|<(1+\epsilon_0)|y|
 \]
 for $(x,y)\in U_R^+$.
Let $0<\epsilon<(1-\epsilon_0)\rho_2 R - \left(\left(|\gamma|+|\kappa|\right)/R\right)$ small enough, we claim that 
\[
\mathcal{B}_1=\{(x,y)\in \mathbb{C}^2: |x|<\epsilon, |y|>2\rho_1\} \subset \mathcal{B}.
\]
We choose $\epsilon$ small enough so that $\mathcal{B}_1\subset U_R^+$.
Let $(x,y)\in \mathcal{B}_1$, then $\Phi_H^+(x,y)=(s,t)\approx (xy+\kappa,y)$. Clearly, $|t|>\rho_1$ and 
\begin{align*}
    |xy+\kappa|\leq \epsilon|y|+|\kappa| & <\left((1-\epsilon_0)\rho_2 R - \frac{|\gamma|+|\kappa|}{R}\right)|y|+|\kappa|\\
    & <(1-\epsilon_0)\rho_2R|y|-|\gamma|<\rho_2R|t|-|\gamma|\\
    & <\rho_2|t|^2-|\gamma|.
\end{align*}
Thus $f\left(\mathcal{B}_1\right)\subset U_R^+$.

\medskip
\noindent
{\it Step 4:}
The Green's function $G_H^+$ has logarithmic growth near point at infinity $[0:1:0]$. Thus, for some $M_1>0$ there exists large $R>0$ such that
\[
\log^+|y|-M_1 \leq G_H^+(x,y)\leq \log^+|y|+M_1,
\] 
on $ V_R^+$. Let $z=(x,y) \in \mathcal{B}_1$, then since $f\left(\mathcal{B}_1\right)\subset U_R^+\subset V_R^+$,
\[
\log^+|f_2(z)|-M_1 \leq G_H^+(f(z)) = G_H^+(z) \leq \log^+|y|+M_1.
\] 
This implies,  $|f_2(z)|\leq e^{2 M_1}|y|$. Further, since
 \[
 |f_1(z)| \leq |f_2(z)|\leq e^{2M}|y|,
 \] 
 we can conclude that $f$ has polynomial growth in $\mathcal{B}_1$. Now we can write 
\begin{equation}\label{f_1exp}
f_1(x,y)=\sum_{n=0}^{\infty}\beta'_n(y)x^n,
\end{equation}
where for each $n\geq 0$, $\beta'_n(y)$ is holomorphic in $|y|>2\rho_2$. If we fix a $y$ and let $r<|y|$, then the Cauchy estimate gives
\[
|\beta'_n(y)|\leq \frac{C_1|y|}{r^n},
\]
for some $C_1>0$.
Letting $r\to |y|$, we get
\[
\beta'_n(y)\leq \frac{C_1}{|y|^{n-1}}.
\]
Rearranging the terms of \eqref{f_1exp}, we get 
\[
f_1(x,y)=a_1 x+b_1 y+c_1+\sum_{n=1}^{\infty}\frac{\alpha_n(x)}{y^n},
\]
where $a_1,$ $b_1$, $c_1 \in \mathbb{C}$ and for every $n\geq 1$, $\alpha_n$ is a polynomial in $x$ of degree at most $n+1$. 
Similarly we can prove
\[
f_2(x,y)=a_2 x+b_2 y+c_2+\sum_{n=1}^{\infty}\frac{\beta_n(x)}{y^n},
\]
where $a_2,$ $b_2$, $c_2 \in \mathbb{C}$ and for every $n\geq 1$, $\beta_n$ is a polynomial in $x$ of degree at most $n+1$. 
Thus, for $(x,y)\in \mathcal{B}_1$, 
\begin{equation}\label{eq5.1}
f(x,y)=\left(a_1 x+b_1 y+c_1+\sum_{n=1}^{\infty}\frac{\alpha_n(x)}{y^n},a_2 x+b_2 y+c_2+\sum_{n=1}^{\infty}\frac{\beta_n(x)}{y^n}\right).
\end{equation}

\subsection*{Proof of Theorem \ref{Thm1}:}
Let $g=f\circ H \circ f^{-1}\circ H^{-1}$. Note that $g(U_H^+)=U_H^+$, and $L$ induces identity on the fundamental group $\pi_1(U_H^+)=\mathbb{Z}[1/d]$. Thus, $g$ is of the form \eqref{eq5.1}. Since $g\in \mathrm{Aut}(\mathbb{C}^2)$, $g$ is an affine map. Therefore, 
$$
 f \circ H= g \circ H \circ f
$$
on $\mathbb{C}^2$. Thus by the main theorem of \cite{CD}, we conclude that $f$ is a polynomial automorphism. Further, by \cite[Theorem 1.2]{Pal1} and \cite[Proposition 3.1]{BPV1}, there exist $s_1,s_2\in \mathbb{Z}$ and a linear map $L$ of the form $(x,y) \rightarrow (a x,a^d y)$ with $a^{d^2-1}=1$ such that $f^{s_1}=L\circ H^{s_2}$.

\section{Proof of Corollaries}
\subsection{Proof of Corollary(\ref{Cor1})} Consider $g= f^{-1}\circ F\circ f \in \rm{Aut}(\mathbb{C}^2)$ then $g(U_H^+)=U_H^+$. Hence from Theorem(\ref{Thm1}) it follows that $g^{s_1}=L\circ H^{s_2}$ for some $s_1,s_2\in \mathbb{Z}$ and linear map $L(x,y)=(ax,a^d y)$ with $a^{d^2-1}=1$. Thus $f^{-1}\circ F^{s_1}\circ f=L\circ H^{s_2}$.
If $s_1>0$ and $s_2<0$, then for every $n\geq 1$,
\[
f^{-1} \circ F^{ns_1} \circ f= (L\circ H^{s_2})^n.
\]
Let $(x,y)\in K_H^+\backslash K_H^-$. Since $f(x,y)\in K_F^+$ and $f$ is continuous, $\{f^{-1}\circ F^{ns_1} \circ f(x,y)\}_{n\geq 0}$ is bounded. But since $G_{L \circ H^{s_2}}^+ \equiv G_{H^{s_2}}^+ \equiv (\mathrm{deg}(H))^{-s_2}G_H^-$, $(x,y)\notin K_{L\circ H^{s_2}}^+$. Thus $\{(L\circ H^{s_2})^n(x,y)\}{n\geq 0}$ is unbounded. Therefore, $s_1>0$ and $s_2<0$ is not possible. Similarly, $s_1<0$ and $s_2>0$ is not possible. Further, if $s_1=0$, then $L\circ H^{s_2}\equiv \mathrm{Id}$. Thus $s_2=0$ and $L\equiv \mathrm{Id}$. Similarly, $s_2=0$ implies $s_1=0$.

\subsection{Proof of Corollary(\ref{Cor2.1}):} The proof we present here follows the same idea as in \cite[Theorem A]{CD}. Note that the support of closed positive $(1,1)$ current $f^*(\mu^+_H)$ is $f^{-1}(J_H^+)\subset K_H^+$ and there exists a unique positive closed $(1,1)$ current with support in $K_H^+$ up to multiplication by a positive constant. Thus, for some $\mathcal{C}>0$, 
\[
f^*(\mu_H^+)=\mathcal{C}\mu_H^+,
\]
which gives that the function 
\[
\mathcal{H}=\mathcal{C}G_H^+-G_H^+\circ f,
\]
is pluriharmonic on $\mathbb{C}^2$. So, there exists a holomorphic function $\mathcal{G}$ on $\mathbb{C}^2$ such that real part of $\mathcal{G}$ is $\mathcal{H}$. We claim that $\mathcal{H}\equiv0$. If not, then $\mathcal{G}\not\equiv 0$. And $J_H^+\subset\{\mathcal{H}=0\}=\cup_{r \in \mathbb{R}}\{\mathcal{G}=i r\}$. Let $q$ be a saddle point of $H$, then since $\{q\}$ is compact hyperbolic set for $H$, by the \textit{stable manifold theorem} the stable set 
\[
W^s(q)=\{z\in \mathbb{C}^2 : \lim_{n\to \infty} H^n(z)=q\}
\]
is an immersed complex submanifold biholomorphic to $\mathbb{C}$. Further, \cite[Theorem 1]{BS2} gives $\overline{W^s(q)}=J_H^+$. Thus, $W^s(q)\subset \cup_{r \in \mathbb{R}}\{\mathcal{G}=i r\}$, and hence for some $r\in \mathbb{R}$, $W^s(q)$ is an irreducible component of $\{\mathcal{G}=i r\}$. Hence $W^s(q)$ is closed in $\mathbb{C}^2$. This gives a contradiction since $J_H^+$ is not $C^1-$smooth (see \cite{BK}). Thus we get $G_H^+\circ f = \mathcal{C} G_H^+$. Thus, $f(K_H^+)=K_H^+$ and the result follows from Theorem \ref{Thm1}.

\section{Proof of Theorem(\ref{Thm4})}
Consider a lift  $\tilde{f}:\mathbb{C}\times\mathcal{A}^{c_1}\to \mathbb{C}\times\mathcal{A}^{c_2}$, of $f$, where $\mathcal{A}^{c_1}=\{t\in \mathbb{C}:\vert t\vert >e^{c_1}\}$ and $\mathcal{A}^{c_2}=\{t\in \mathbb{C}:\vert t\vert >e^{c_2}\}$. Then, $\tilde{f}$ is of the form $\tilde{f}(s,t)=\left(\tilde{f}_1(s,t),\tilde{f}_2(t)\right)$, where $\tilde{f}_1(s,t)=\beta(t)s+\gamma(t)$ with $\beta,\gamma:\mathcal{A}^{c_1}\to\mathbb{C}$ holomorphic. Let $\tilde{f}^*(s,t)=\left(\tilde{f}_1^*(s,t),\tilde{f}_2^*(t)\right)$ be the inverse of $\tilde{f}$. Then 
\[
\tilde{f}_2\circ\tilde{f}_2^*=Id|_{\mathcal{A}^{c_1}} \text{\hspace{1mm} and } \tilde{f}_2^*\circ\tilde{f}_2=Id|_{\mathcal{A}^{c_2}}.
\]
Thus $\tilde{f}_2:\mathcal{A}^{c_1}\to\mathcal{A}^{c_2}$ is a biholomorphism. Let $h_1:\mathbb{D}^*\to \mathcal{A}^{c_1}$ and $h_2:\mathcal{A}^{c_2}\to\mathbb{D}^*$ be defined by $h_1(t)=e^{c_1}/t$ and $h_2(t)=e^{c_2}/t$, respectively.
Let $h:\mathbb{D}^*\to \mathbb{D}^*$ is defined by $h(t)=h_2\circ\tilde{f}_2\circ h_1(t)$. Then $h$ a biholomorphism. Define $h(0)=0$. Since $\tilde{f}_2$ has a pole at infinity, $h:\mathbb{D}\to\mathbb{D}$ is an automorphism. Thus, $h(t)=\alpha' t$, for some $|\alpha'|=1$, which implies
\[
\tilde{f}_2(t)=\alpha t,
\]
with $\alpha=e^{c_2-c_1}\alpha'.$
For $k\geq 1$, $n\geq 0$, let $\gamma_{\frac{k}{d^n}}$ be the deck transformation for the cover $\hat{\pi}^{c_1}:\hat \Omega^{c_1}\to \Omega^{c_1}$. Then 
\[
\hat{\pi}^{c_2}\circ\tilde{f}\circ\gamma_{\frac{k}{d^n}}\circ \tilde{f}^{-1}=f\circ\hat{\pi}^{c_1}\circ\gamma_{\frac{k}{d^n}}\circ\tilde{f}^{-1}=f\circ\hat{\pi}^{c_1}\circ\tilde{f}^{-1}=\hat{\pi}^{c_2}.
\]
Thus $\tilde{f}\circ\gamma_{\frac{k}{d^n}}\circ \tilde{f}^{-1}$ is a deck transformation for the cover $\hat{\pi}^{c_2}:\hat \Omega^{c_2}\to \Omega^{c_2}$. Note that the second component of $\tilde{f}\circ\gamma_{\frac{k}{d^n}}\circ \tilde{f}^{-1}$ is $e^{2\pi i k/d^n}t$. Thus $\tilde{f}\circ\gamma_{\frac{k}{d^n}}=\gamma_{\frac{k}{d^n}}\circ \tilde{f}$. Therefore for $(s,t)\in \mathbb{C}\times\mathcal{A}^{c_1}$ \eqref{eq07.1} holds. The same arguments as those in Section~4 show that $\beta(t)\equiv \beta$ with $\beta^{d-1}=1$, and for some $n_0\geq 1$, $\alpha^{d^{n_0}(d+1)}=\beta$. Thus $|\alpha|=1$, which implies $c_1=c_2=c$. Now if we compose $\tilde{f}$ with an appropriate deck transformation then we get a lift $\tilde{F}:\mathbb{C}\times \mathcal{A}^{c}\to \mathbb{C}\times \mathcal{A}^{c}$ of $f$, of the form 
\[
(s,t)\mapsto(\beta s+\gamma,\tilde{\alpha}t),
\]
such that $\beta^{d-1}=1$, $\tilde{\alpha}^{d+1}=\beta$ and $\gamma$ is constant. By abuse of notation, we will write $\tilde{F}=\tilde{f}$ and $\tilde{\alpha}=\alpha$.

\medskip
\noindent
Now for any $c_3>c$,  
\begin{align*}
& \tilde{f}(\mathbb{C}\times \{|t|=e^{c_3}\})=\mathbb{C}\times \{|t|=e^{c_3}\} \\
& \implies \hat \pi^c \circ\hat{\Phi}^{-1}_c\circ\tilde{f}(\mathbb{C}\times \{|t|=e^{c_3}\})=\hat{\pi}^c\circ\hat{\Phi}^{-1}_c(\mathbb{C}\times \{|t|=e^{c_3}\})\\
& \implies f\circ \hat{\pi}^c\circ\hat{\Phi}^{-1}_c(\mathbb{C}\times \{|t|=e^{c_3}\})=\hat{\pi}^c\circ\hat{\Phi}^{-1}_c(\mathbb{C}\times \{|t|=e^{c_3}\})\\
&\implies f(\{G_H^+=c_3\})=\{G_H^+=c_3\}\\
& \implies G_H^+ \circ f \equiv G_H^+ \text{ \hspace{1mm} on } \Omega^c.
\end{align*}
Thus, for every $c_3>c$,
\[
f|_{\Omega^{c_3}}:\Omega^{c_3}\to \Omega^{c_3}
\]
is an automorphism.
Let $\mathcal{B}_{1,c}=\mathcal{B}_1\cap \Omega^c$. Following similar arguments as in {\it Step~3} of the proof of Theorem(\ref{Thm1}), we get $f(\mathcal{B}_{1,c})\subset U_R^+$. Now adopting arguments from {\it Step~4} of the proof of Theorem(\ref{Thm1}), we can prove that $f$ has polynomial growth in $\mathcal{B}_{1,c}$, which further gives
\[
f(x,y)=\left(a_1 x+c_1 + \sum_{n=1}^{\infty}\frac{\alpha_n(x)}{y^n},a_2 x+b_2 y+c_2+\sum_{n=1}^{\infty}\frac{\beta_n(x)}{y^n}\right),
\]
on $\mathcal{B}_{1,c}$, with $a_1,a_2,b_1,b_2,c_1,c_2\in \mathbb{C}$ and for $n\geq 1$, $\alpha_n$ and $\beta_n$ polynomials in $x$ of degree at most $n+1$.

\medskip
\no
Now, let $f\in \mathrm{Aut}(\mathbb{C}^2)$. Since $G_H^+\circ f= G_H^+$ on $\Omega^c\implies G_H^+\circ f= G_H^+$ on $\mathbb{C}^2$. Thus $f(U_H^+)=U_H^+$ and the result follows from \ref{Thm1}.

\bibliographystyle{amsplain}
%\bibliography{ref_rigidity}

\end{document}